\documentclass[12pt]{amsart}
\usepackage{amssymb,latexsym, array}
\usepackage{verbatim,euscript}
\usepackage{fullpage,color}
\usepackage{tikz}
\usetikzlibrary{arrows,shapes,trees}

\usepackage[colorlinks=true,linkcolor=blue,urlcolor=my_color,citecolor=magenta]{hyperref}

\definecolor{my_color}{rgb}{0,0.5,0.5}
\definecolor{MIXT}{rgb}{0.4,0.3,0.6}
\definecolor{mixt}{rgb}{0.5,0.3,0.2}
\definecolor{sin}{rgb}{0,0.5,0.5}
\definecolor{darkblue}{rgb}{0,0.1,0.8}
\definecolor{redi}{rgb}{0.5,0,0.4}

\input {cyracc.def}
\numberwithin{equation}{section}

\font\tencyr=wncyr10 
\def\rus{\tencyr\cyracc}

\newtheorem{thm}{Theorem}[section]
\newtheorem{lm}[thm]{Lemma} 

\newtheorem{prop}[thm]{Proposition}

\theoremstyle{remark}

\newtheorem{rmk}[thm]{Remark}

\theoremstyle{definition}
\newtheorem{ex}[thm]{Example}
\newtheorem{df}{Definition}

\newcommand {\be}{{\mathfrak b}}

\newcommand {\g}{{\mathfrak g}}

\newcommand {\el}{{\mathfrak l}}

\newcommand {\p}{{\mathfrak p}}
\newcommand {\q}{{\mathfrak q}}
\newcommand {\rr}{{\mathfrak r}}
\newcommand {\es}{{\mathfrak s}}
\newcommand {\te}{{\mathfrak t}}

\newcommand {\gln}{{\mathfrak{gl}}_n}
\newcommand {\sln}{{\mathfrak{sl}}_n}

\newcommand {\gltn}{{\mathfrak{gl}}_{2n}}

\newcommand {\gltno}{{\mathfrak{gl}}_{2n+1}}

\newcommand {\spn}{{\mathfrak{sp}}_{2n}}
\newcommand {\son}{{\mathfrak{so}}_{n}}
\newcommand {\sono}{{\mathfrak{so}}_{2n+1}}
\newcommand {\sone}{{\mathfrak{so}}_{2n}}

\newcommand {\esi}{\varepsilon}
\newcommand {\ap}{\alpha}

\newcommand {\ads}{{\mathrm{ad^*}}}

\newcommand {\ind}{{\mathsf{ind\,}}}
\newcommand {\Lie}{{\mathrm{Lie\,}}}

\newcommand {\rk}{{\mathsf{rk\,}}}

\newcommand {\un}{\underline}
\newcommand {\tri}{\mathfrak{sl}_2}

\newcommand {\beq}{\begin{equation}}
\newcommand {\eeq}{\end{equation}}

\newcommand {\boa}{\un{a}}
\newcommand {\koa}{k(\un{a})}
\newcommand {\bob}{\un{b}}
\newcommand {\MRS}{{\sf MRS\ }}

\newcommand{\gt}{\mathfrak}

\newcommand{\GL}{{\rm GL}}

\newcommand {\cK}{{\mathcal K}}

\newcommand {\cT}{{\mathcal T}}

\newcommand {\BC}{{\mathbb C}}

\newcommand {\BR}{{\mathbb R}}

\renewcommand{\le}{\leqslant}
\renewcommand{\ge}{\geqslant}

\newfam\eusfam%
\font\euszw=eusm10 scaled 1200%
\font\eusac=eusm7 scaled 1200%
\font\eusacc=eusm7 scaled 1000%
\textfont\eusfam=\euszw\scriptfont\eusfam=\eusac%
\scriptscriptfont\eusfam=\eusacc%
\begin{document}
\hfill {\scriptsize February 24, 2017}
\vskip1ex

\title{On seaweed subalgebras and meander graphs in type {\sf D}}
\author[D.\,Panyushev]{Dmitri I.~Panyushev}
\address[D.P.]{Institute for Information Transmission Problems of the Russian Academy of Sciences, 
Bolshoi Karetnyi per. 19,  Moscow 127051, Russia}
\email{panyushev@iitp.ru}
\author[O.\,Yakimova]%
{Oksana S.~Yakimova}
\address[O.Y.]{Institut f\"ur Mathematik, Friedrich-Schiller-Universit\"at Jena,  D-07737 Jena, 
Deutschland}
\email{oksana.yakimova@uni-jena.de}
\thanks{The research of the first author was carried out at the IITP RAS at the expense of the Russian Foundation for Sciences (project {\rus N0} 14-50-00150). 
The second author is partially supported by the Graduiertenkolleg GRK 1523 
``Quanten- und Gravitationsfelder".}


\begin{abstract}
In 2000, Dergachev and Kirillov introduced subalgebras of "seaweed type" in $\gln$ and computed their 
index using certain graphs, which we call type-{\sf A} meander graphs. Then the subalgebras of seaweed 
type, or just  "seaweeds", have been defined by Panyushev (2001) for arbitrary reductive Lie algebras. 
Recently, a meander graph approach to computing the index in types {\sf B} and {\sf C} has been developed by the authors.
In this article, we consider the most difficult and interesting case of type ${\sf D}$. Some new phenomena occurring here are related to the fact that the Dynkin diagram has a branching node.
\end{abstract}
\maketitle

\section{Introduction}   
\label{sect:intro}

\noindent 
A general philosophy of Representation Theory proclaims that the coadjoint action of an algebraic group 
$Q$ encodes information on many other actions.   
An important numerical characterisation of the  coadjoint action is 
the index. 

The {\it index\/} of an algebraic Lie algebra $\q$, $\ind\q$, is the minimal dimension of the stabilisers for the 
coadjoint representation of $\q$.  If $\q$ is reductive, then $\ind\q=\rk\q$. Hence the index can be thought of 
as a generalisation of rank.  But for non-reductive Lie algebras, it is often hard to evaluate.  In this paper,
we elaborate on the meander graph approach to computing the index of the seaweed subalgebras in 
$\sone$. As similar method have previously been developed in types {\sf A,B}, and {\sf C}~\cite{dk00,SW-C}, our present results complete a meander graph approach to the index of seaweed 
subalgebras for the classical Lie algebras.

For $\gln$, the seaweed subalgebras (or just  {\it seaweeds}) have been introduced by Dergachev and 
Kirillov~\cite{dk00}.  These are subalgebras of specific matrix shape (see Figure~\ref{fig:type-A} below) that 
resembles seaweeds, hence the term.  A general definition suited for arbitrary reductive Lie algebras 
$\g$ appears in~\cite{Dima01}. Namely, if $\p_1,\p_2\subset \g$ are parabolic subalgebras such that 
$\p_1+\p_2=\g$, then $\q=\p_1\cap\p_2$ is called a seaweed in $\g$. (For this reason, some people began 
to use later the term "biparabolic subalgebra" for such $\q$.) 
The seaweed subalgebras form a wide class of Lie algebras which include all parabolics and their Levi 
subalgebras. 

Without loss of generality, one may assume that  $\p_1$ and $\p_2$ are ``adapted'' to a fixed 
triangular decomposition of $\g$, see Section~\ref{sect:generalities} for details. Then $\q$ is said to be 
{\it standard}. The standard seaweeds are in a one-to-one correspondence with the pairs of subsets of the set of simple roots of $\g$~\cite{Dima01}. 
An inductive procedure for computing the index of standard seaweeds in the classical Lie algebras is 
presented in~\cite{Dima01}. The procedure helps to answer several subtle 
questions on the coadjoint action~\cite{Dima03-b,MY12}. In that procedure, seaweeds naturally appear 
when one is trying to compute the index of a parabolic subalgebra in type {\sf A}.   
In the other classical types, a parabolic cannot be reduced any further and therefore the parabolic 
subalgebras have to be included into the induction base. In types {\sf B} and {\sf C}, 
any seaweed can be reduced to a parabolic. However, this is not always the case in type~{\sf D}, and this
phenomenon was overlooked in~\cite[Sect.\,5]{Dima01}. 
This is also one of the sources of many difficulties in developing 
the theory of meander graphs in type {\sf D}. 

In \cite{dk00}, the index of the seaweed subalgebras of $\gln$ has been computed using certain graphs,
which are said to be type-{\sf A} meander graphs. 
Recently, the authors introduced meander graphs in types {\sf C} and {\sf B},
and gave a formula for the index of seaweeds in terms of these graphs~\cite{SW-C}. 
In this paper, we introduce {\it type-{\sf D} meander graphs\/} and compute the index of the seaweeds 
in $\sone$ via these graphs, see Theorem~\ref{thm:main}. 
Unlike the other classical cases, the present situation is more involved, and the reason is that the Dynkin diagram of {\sf D}$_n$ has a branching node. Thanks to the presence of the branching node, we get two new phenomena. {\it First}, there is no natural bijection between the standard parabolics in $\sone$ 
and the compositions with sum at most $n$ (as it happens in $\sono$ and $\spn$). {\it Second}, there are 
certain seaweeds in $\sone$ that do not admit a matrix realisation of ``seaweed shape". The definition of a meander graph for them requires a trick, and their meander graphs acquire two arcs crossing each other.
This is a completely new phenomenon that does not occur in the other classical types. The corresponding subalgebras are said to be {\it seaweeds with crossing}. The seaweeds that cannot be reduced to a 
parabolic occur only among seaweeds with crossing.
 
A general algebraic formula for the index of the seaweeds has been proposed 
in~\cite[Conj.\,4.7]{ty-AIF} and then proved in~\cite[Section\,8]{jos}. An advantage of the meander graph 
approach is that it allows us to detect new interesting classes of Frobenius seaweeds. Recall that $\q$ is 
called {\it Frobenius\/} if $\ind\q=0$.
These are curious Lie algebras that gained popularity 
owing to their connection with the classical Yang-Baxter equation.  
For more on Frobenius Lie algebras and their r\^ole in Invariant Theory, see e.g.~\cite{Ooms}.  

Further properties of the coadjoint action have been studied for the seaweed algebras. For instance, articles~\cite{j2,MY12} show that there are many interesting phenomena arising here. 

The structure of the article is as follows.
In Section~\ref{sect:generalities}, we provide generalities on the arbitrary seaweeds and recall known 
results in types {\sf A,B,} and {\sf C}. Section~\ref{sect:D} is devoted to the detailed 
construction of meander graphs for the seaweeds in $\sone$. Our main result---a formula for $\ind\q$ in 
terms of the meander graph $\Gamma_n(\q)$---is stated and proved in Section~\ref{sect:main}. As 
in~\cite{SW-C}, our proof heavily relies on the inductive procedure of \cite{Dima01}.
In Section~\ref{sect:applic}, we gather some further results
concerning generic stabilisers, maximal reductive stabilisers, and  Frobenius cases for seaweeds in 
$\sone$.
\\ \indent
Throughout the paper, the  ground field is $\BC$. 

\section{Generalities on seaweed subalgebras and meander graphs}   
\label{sect:generalities}

\noindent 
We assume that a reductive algebraic Lie algebra $\g$ is equipped with a fixed triangular decomposition, 
so that there are two opposite 
Borel subalgebras $\be$ and $\be^-$, and a Cartan subalgebra $\te=\be\cap\be^-$. Let $\Delta^+$ be the 
set of roots of $(\be,\te)$ and $\Pi=\Pi_\g=\{\ap_1,\dots,\ap_n\}$ the set of simple roots in $\Delta^+$.
If $\gamma\in\Delta$, then  $\g_\gamma$ is the corresponding root space. 

Let $\p_1$ and $\p_2$ be two parabolic subalgebras of $\g$. If $\p_1 +\p_2=\g$, then $\p_1\cap\p_2$ is 
called a {\it seaweed subalgebra\/} or just {\it seaweed\/} in $\g$ (see \cite{Dima01}).
The set of seaweeds includes all parabolics (if $\p_2=\g$), all Levi subalgebras (if $\p_1$ and $\p_2$ are 
opposite), and many interesting non-reductive subalgebras. Without loss of generality,  we may also 
assume that  $\p_1\supset \be$ (i.e., $\p_1$ is {\it standard}) and $\p_2=\p_2^-\supset \be^-$ 
(i.e., $\p_2$ is {\it opposite-standard}). Then the seaweed $\q=\p_1\cap\p_2^-$ is said to be {\it standard}, 
too. Either of these parabolics is determined by a subset of $\Pi$. If $\p$ is standard, $\el\subset\p$ is a 
standard Levi subalgebra (i.e. $\te\subset\el$), and $S\subset\Pi$ is the set of simple roots of $\el$,
then we write $\el=\el(S)$ and $\p=\p(S)$; and likewise for $\p^-\supset\be^-$. 
In particular, $\p(\varnothing)=\be$, $\p^-(\varnothing)=\be^-$, and $\p(\Pi)=\p^-(\Pi)=\g$.
Then $\Pi\setminus S=\{\ap\in\Pi \mid \g_{-\ap}\not\subset \p(S)\}=\{\ap \mid \g_{\ap}\not\subset \p^-(S)\}$. 
Thus, a standard seaweed is determined by two arbitrary subsets  $S,T\subset \Pi$, and
we set  $\q(S,T)=\p(S)\cap \p^-(T)$, cf. also~\cite[Sect.\,2]{Dima01}. 
Clearly, $\q(S,T)$ is reductive if and only if $S=T$, $\q(S,T)$ is parabolic if and only if $S=\Pi$ or
$T=\Pi$, and $\q(S,T)\simeq \q(T,S)$.

\begin{rmk}   \label{rem:union=Pi}
If $\ap\in \Pi\setminus (S\cup T)$ , then $\q(S,T)$ is contained in the Levi subalgebra
$\el(\Pi\setminus\{\ap\})$.
Therefore, $\q(S,T)$ does not belong to a proper Levi if and only if $S\cup T=\Pi$.
\end{rmk}

\subsection{Compositions and meander graphs in type {\sf A}} \label{subs:A}
Let us recall the construction of meander graphs in type {\sf A}. It is more convenient here to work with 
$\gln$ in place  of $\sln$. A {\it composition\/} is a finite sequence of positive integers, say 
$\un{a}=(a_1,\dots,a_s)$. Set $|\un{a}|=\sum_i a_i$ and $\un{a}^{-1}=(a_s,\dots,a_1)$. 
We say that $\un{a}$ is a composition of $m$, if $|\un{a}|=m$. 

We work with the obvious triangular decomposition of $\gln$, where $\be$ consists of the 
upper-triangular matrices. If $S=\{\ap_{i_1},\dots,\ap_{i_{s-1}}\}$ and $\p(S) \supset\be$, then $\el(S)$ 
has the consecutive diagonal blocks 
$\mathfrak{gl}_{a_1}$, $\mathfrak{gl}_{a_2}$, \dots ,$\mathfrak{gl}_{a_s}$,
where $a_j=i_j-i_{j-1}$ with $i_0=0$, and $i_{s}=n$.
Then we write $\p(S)=\p(\un{a})$ and $\el(S)=\el(\un{a})$, where $\un{a}=(a_1,a_2,\dots,a_s)$. 
In particular, $\be=\p(1,\dots,1)$ and $\gln=\p(n)$.
Note that all $a_i \ge 1$ and $|\un{a}|=n$. Likewise, if $\p^-(T)\supset\be^-$ is similarly represented by 
the composition $\un{b}=(b_1,\dots,b_t)$ with $|\un{b}|= n$, then $\p^-(T)=\p^-(\un{b})$ and the standard 
seaweed $\q(S,T)\subset \gln$ is denoted by $\q^{\sf A}(\un{a}{\mid} \un{b})$.  A sample picture
is given in Fig.~\ref{fig:type-A}.

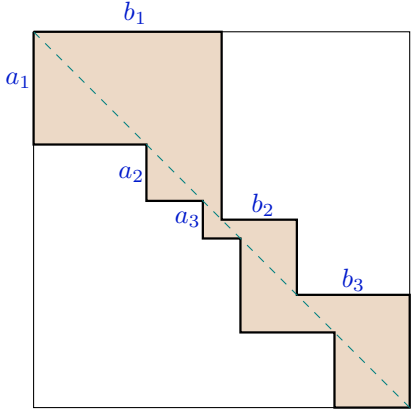
\begin{figure}[htb]
\begin{center}
\begin{tikzpicture}[scale= .5]
\draw (0,0)  rectangle (10,10);
\path[fill=brown!30,draw,line width=.3mm]  (0,10) -- (0,7) -- (3,7) -- (3,5.5) -- (4.5,5.5) -- (4.5,4.5) -- (5.5,4.5) -- (5.5,2) -- (8,2) -- (8,0)-- (10,0) -- (10,3) -- (7,3) -- (7,5) -- (5,5) -- (5,10) -- (0,10); 

\draw (-0.4,8.7)  node {\footnotesize {\color{darkblue}$a_1$}} ;
\draw (2.6,6.2)  node {\footnotesize {\color{darkblue}$a_2$}} ;
\draw (4.1,5)  node {\footnotesize {\color{darkblue}$a_3$}} ;
\draw (2.7,10.5)  node {\footnotesize {\color{darkblue}$b_1$}} ;
\draw (6.1,5.4)  node {\footnotesize {\color{darkblue}$b_2$}} ;
\draw (8.5,3.4)  node {\footnotesize {\color{darkblue}$b_3$}} ;
\draw[dashed,my_color]  (10,0) -- (0,10) ;
\end{tikzpicture}
\caption{A standard seaweed subalgebra of $\gln$} \label{fig:type-A}
\end{center}
\end{figure}
\noindent
The corresponding {\it type-{\sf A} meander graph\/} $\Gamma=\Gamma^{\sf A}(\un{a}{\mid}\un{b})$ 
is defined by the following rules:
\par\textbullet
\quad
$\Gamma$ has $n$ consecutive vertices on a horizontal line numbered from $1$ up to $n$.
\par\textbullet
\quad The parts of $\un{a}$ determine the set of pairwise disjoint arcs (edges) that are drawn \un{\sl below\/} the horizontal line.
Namely, part $a_1$ determines $[a_1/2]$ consecutively embedded arcs below the 
nodes $1,\dots,a_1$, where the widest arc joins vertices 1 and $a_1$, the following joins $2$ and $a_1-1$, etc. If $a_1$ is odd, then the middle vertex $(a_1+1)/2$ acquires no arc at all. 
Next,  part $a_2$ determines $[a_2/2]$ embedded
arcs below the nodes $a_1+1,\dots,a_1+a_2$, etc.
\par\textbullet
\quad The arcs corresponding to $\un{b}$ are drawn by the same rules,
but \un{\sl above\/} the horizontal line. 

It follows that the degree of each vertex in $\Gamma$ is at most $2$ and each connected component of 
$\Gamma$ is homeomorphic to either a circle or a segment. (An isolated vertex is also a segment!) 
By~\cite{dk00}, the index of $\q^{\sf A}(\un{a}{\mid} \un{b})$ is determined via
$\Gamma=\Gamma^{\sf A}(\un{a}{\mid}\un{b})$ as follows: 
\beq  \label{eq:index-A}
\ind\q^{\sf A}(\un{a}{\mid} \un{b})=2{\cdot}\text{(number of cycles in $\Gamma$)} + 
\text{(number of segments in $\Gamma$)} .
\eeq
Clearly, $\q^{\sf A}(\un{a}{\mid} \un{b})\simeq \q^{\sf A}(\un{b}{\mid} \un{a})\simeq
\q^{\sf A}(\un{a}^{-1}{\mid} \un{b}^{-1})$ and the corresponding graphs are also isomorphic.
For $\un{b}=(n)$, we obtain the meander graph for the parabolic $\p(\un{a})$; whereas for $\un{a}=(n)$, we get the meander graph for the parabolic $\p^-(\un{b})$.

\begin{rmk} \label{rem:ind-gl-sl}
Formula~\eqref{eq:index-A} gives the index of a seaweed in $\gln$, not in $\sln$. However, if 
$\q\subset\gln$ is a seaweed, then $\q\cap\sln$ is a seaweed in $\sln$ and the mapping 
$\q\mapsto \q\cap\sln$ is a bijection. Here $\q=(\q\cap\sln)\oplus (\text{1-dim centre of $\gln$})$, hence
$\ind(\q\cap\sln)=\ind\q-1$. Since $\ind\q^{\sf A}(\un{a}{\mid} \un{b})\ge 1$ and the minimal value `1' is 
achieved if and only if $\Gamma$ is a sole segment, we also obtain a characterisation of the Frobenius seaweeds in $\sln$ via meander graphs. 
\end{rmk}
\begin{ex}   \label{ex:old-A} 
$\Gamma^{\sf A}(2,4,3{\mid}  5,4)$=
\setlength{\unitlength}{0.021in}
\raisebox{-12\unitlength}{%
\begin{picture}(90,30)(-2,-11)
\multiput(0,3)(10,0){9}{\circle*{2}}

{\color{blue}
\put(20,5){\oval(40,20)[t]}
\put(20,5){\oval(20,10)[t]}
\put(65,5){\oval(10,6)[t]}
\put(65,5){\oval(30,15)[t]}

\put(5,1){\oval(10,6)[b]}
\put(35,1){\oval(30,15)[b]} 
\put(35,1){\oval(10,6)[b]}
\put(70,1){\oval(20,10)[b]}
}
\end{picture}
}  and the index of the corresponding seaweed in $\mathfrak{gl}_9$ (resp. $\mathfrak{sl}_9$) 
equals $3$ (resp. $2$).
\end{ex}

\subsection{Compositions and meander graphs in types {\sf B} and {\sf C}} 
\label{subs:BC}
For ${\sf B}_n$ or ${\sf C}_n$, {\bf any} standard parabolic $\p\subset\g$ has the standard Levi in 
the block-diagonal form, in the appropriate matrix realisation of $\g$. This associates a certain 
composition $\un{a}$ with $|\un{a}|\le n$ to $\p$, and this correspondence appears to be a bijection.
 
The idea that works fine 
for $\spn$ is that, for a standard parabolic 
$\p=\p(\un{a})\subset\spn$, one considers the ``symmetric'' composition 
$\un{\tilde a}:=(\un{a}, 2(n-|\un{a}|), \un{a}^{-1})$ of $2n$ and 
the ``symmetric'' parabolic $\tilde \p=\p(\un{\tilde a})$ in $\gltn$. Then $\p=\tilde\p\cap \spn$ and 
likewise for the opposite-standard parabolics. The {\it type-{\sf C} meander graph} of the seaweed 
$\p_1\cap\p_2^-\subset\spn$ is defined via the type-{\sf A} meander graph of
$\tilde\p_1\cap\tilde\p_2^-\subset\gltn$. 
Namely, letting $\Gamma^{\sf C}(\p_1\cap\p_2^-)=\Gamma^{\sf A}(\tilde\p_1\cap\tilde\p_2^-)$, 
we obtain a graph with $2n$ vertices, which is symmetric w.r.t. the middle. The symmetry w.r.t. the middle is denoted by $\sigma$. Then 
\beq  \label{eq:index-C}
\ind (\p_1\cap\p_2^-)=\#\{\text{the cycles of}\ \Gamma^{\sf C}\} +
\frac{1}{2}\#\{\text{the non-$\sigma$-stable segments of $\Gamma^{\sf C}$}\} ,
\eeq
see~\cite[Theorem\,3.2]{SW-C} for the details. With minor adjustments, this works for $\sono$, too.
\\ \indent
Because our type-{\sf B} explanations in \cite{SW-C} are rather sketchy, we provide an intrinsic 
construction of the type-{\sf B} meander graphs. (This is going to be helpful for our next exposition in type {\sf D}, where some difficulties occur.)
We think of $\sono$ as the set of skew-symmetric matrices w.r.t{.} the antidiagonal. 
The triangular decomposition of $\sono$ is induced by that of $\gltno$, and
we deal with the usual numbering of simple roots, so that $\ap_i=\esi_i-\esi_{i+1}$ for $i< n$ and 
$\ap_n=\esi_n$. If  $\Pi\setminus S=\{\ap_{i_1},\dots, \ap_{i_s}\}$, then the consecutive diagonal blocks 
of the standard Levi $\el(S)$ are 
\\[.6ex]    \centerline{
$\gt{gl}_{i_1}$, $\gt{gl}_{i_2-i_1}$,\dots, $\gt{gl}_{i_s-i_{s-1}}$, 
$\gt{so}_{2(n-i_s)+1}$, and then {\it the same $\gt{gl}$-blocks in the reverse order}.}
\\[.6ex]
\noindent
Here (and below) the words "the same $\gt{gl}$-blocks" refer not only to the size, but also to the fact that the resulting matrices must be skew-symmetric w.r.t. the antidiagonal.

The associated composition is $\un{a}=(i_1,i_2-i_1,\dots,i_s-i_{s-1})$ with  $|\un{a}|=i_s\le n$, and we 
also write
$\p(\un{a})$ for $\p(S)$. If $S=\Pi$ and $\p(\Pi)=\sono$, then the associated composition is empty, with sum $0$. This yields a {\bf bijection} 
between the standard parabolics and the compositions with sum at most $n$. Consequently, any 
standard seaweed in $\sono$ has a symmetric (w.r.t. the antidiagonal) ``seaweed shape'' and can be 
identified with a pair of compositions. (See Fig.~\ref{fig:sono}, where $\un{a}=(a_1,a_2)$ and 
$\un{b}=(b_1)$.)  
\\ \indent
To define the type-{\sf B} meander graphs, we use the embedding $\sono\subset\gltno$. For a standard 
parabolic $\p=\p(\un{a})\subset\sono$, consider the ``symmetric'' composition 
$\un{\hat a}:=(\un{a}, 2(n-|\un{a}|)+1, \un{a}^{-1})$ of $2n+1$ and the ``symmetric'' parabolic 
$\hat \p=\p(\un{\hat a})$ in $\gltno$. Then $\p=\hat\p\cap \sono$ and 
likewise for the opposite-standard parabolics. The {\it type-{\sf B} meander graph} of the seaweed 
$\p_1\cap\p_2^-\subset\sono$ is defined via the type-{\sf A} meander graph of
$\hat\p_1\cap\hat\p_2^-\subset\gltno$. 
Namely, letting $\Gamma^{\sf B}(\p_1\cap\p_2^-)=\Gamma^{\sf A}(\hat\p_1\cap\hat\p_2^-)$, 
we obtain a graph with $2n+1$ vertices, which is symmetric w.r.t. the middle. Then Eq.~\eqref{eq:index-C}
remains true in type {\sf B}, with the same proof. Since the middle part of the symmetric composition $\un{\hat a}$ is odd, the middle vertex $n+1$ remains isolated in $\Gamma^{\sf A}(\dots)$ for {\bf all} seaweeds in $\sono$.
It is also a $\sigma$-stable segment, which is not counted in the {\sf B}-analogue of 
Eq.~\eqref{eq:index-C}. Therefore, this middle vertex can safely be removed from the type-{\sf B} meander graphs, which yields exactly the same graphs as in type {\sf C}. Thus, we arrive at conclusion (1) made
in p.~498 in \cite{SW-C}. But this time we see the reason behind it.

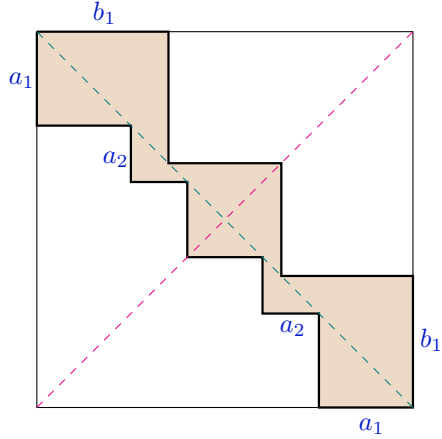
\begin{figure}[htb]
\begin{center}
\begin{tikzpicture}[scale= .5]
\draw (0,0)  rectangle (10,10);
\path[fill=brown!30,draw,line width=.3mm]  (0,10) -- (0,7.5) -- (2.5,7.5) -- (2.5,6) -- (4,6) -- (4,4) -- (6,4) -- (6,2.5) -- (7.5, 2.5) -- (7.5, 0) -- (10,0) -- (10,3.5) -- (6.5,3.5) -- (6.5,6.5) -- (3.5,6.5) -- (3.5,10) -- (0,10); 

\draw (-0.4,8.7)  node {\footnotesize {\color{darkblue}$a_1$}} ;
\draw (2.1,6.6)  node {\footnotesize {\color{darkblue}$a_2$}} ;
\draw (8.9,-0.5)  node {\footnotesize {\color{darkblue}$a_1$}} ;
\draw (6.8,2.1)  node {\footnotesize {\color{darkblue}$a_2$}} ;
\draw (1.8,10.5)  node {\footnotesize {\color{darkblue}$b_1$}} ;
\draw (10.5,1.8)  node {\footnotesize {\color{darkblue}$b_1$}} ;
\draw[dashed,my_color]  (10,0) -- (0,10) ;
\draw[dashed,magenta]  (0,0) -- (10,10) ;
\end{tikzpicture}
\caption{A standard seaweed in $\sono$}
\label{fig:sono}
\end{center}
\end{figure}

However, there is neither such a uniform bijection nor a simple construction of meander graphs in 
type {\sf D}, and the reason is that the Dynkin diagram has a branching node. 

\section{Compositions and meander graphs in type {\sf D}}
\label{sect:D}
\noindent
We think of $\sone$ as the set of skew-symmetric matrices {w.r.t.}~the antidiagonal.
Since $\gt{so}_2\simeq\gt{gl}_1$, $\gt{so}_4\simeq\tri\oplus\tri$, and $\gt{so}_6\simeq\gt{sl}_4$, we may assume that $n\ge 4$. However, these small rank cases may appear in our future reduction procedure.
The triangular decomposition of $\g=\sone$ is induced by that of $\gltn$. In particular, 
\\[.6ex]
\centerline{
$\be=\sone\cap \{ \text{the upper-triangular matrices in $\gltn$}\}$
}
\\[.6ex]  \noindent
is the fixed Borel subalgebra of $\sone$ and 
$\te=\{\mathsf{diag}(\esi_1,\dots,\esi_n,-\esi_n,\dots,-\esi_1)\}$. Then
$\ap_i=\esi_i-\esi_{i+1}$ for $i<n$, and $\ap_n=\esi_{n-1}+\esi_n$. 
 
\subsection{Parabolic subalgebras and compositions} 
\label{subs:parab-comp}
The first trouble is that if $\p=\p(S)$ or $\p^-(S)$ with $\ap_{n-1}\not\in S$ and $\ap_n\in S$, then $\el(S)$ 
does {\bf not} have a block diagonal matrix form, see Fig.~\ref{fig:non-block-D}. Here 
one can swap $\ap_{n-1}$ and $\ap_n$, which provides an "admissible" subset of $\Pi$ and an isomorphic 
parabolic. This swapping can be understood as changing the matrix realisation of $\sone$. 
But this does not always help in case of seaweeds, i.e., pairs of parabolics. If $\q=\q(S,T)$ is 
such that $\ap_{n-1}\in T\setminus S$ and  $\ap_{n}\in S\setminus T$, then swapping changes nothing 
and $\q$ does not have a ``seaweed shape'',  as in Fig.~\ref{fig:sono}. This phenomenon was overlooked 
in~\cite{Dima01}. 
To realise other possible difficulties, let us consider in more details the interrelation 
between standard parabolics of $\sone$ and compositions.

\begin{prop}   \label{prop:D-parab-and-comp}
Let\/ $\p(S)\subset\sone$ be a standard parabolic. Then
 
{\sf (1)} \ $\p(S)$ does not have a block triangular form if and only if $\ap_{n-1}\not\in S$ and 
$\ap_n\in S$.

{\sf (2)} \ In all other cases, using the block triangular form, one naturally associates to $\p(S)$ a 
composition $\un{a}$ with $|\un{a}|\le n$ and $|\un{a}|\ne n-1$. More precisely, 
\\ \indent
 \quad  {\sf (i)} \ if $\ap_{n-1},\ap_n\in S$, then $|\un{a}|\le n-2$;
\\ \indent   
\quad   {\sf (ii)} \ if $\ap_n\not\in S$, then $|\un{a}|= n$.
\end{prop}
\begin{proof} (1) \ Obvious. E.g. see Fig.~\ref{fig:non-block-D} for $\Pi\setminus S=\{\ap_{n-1}\}$.
\\ \indent
(2i) \ If $\ap_{n-1},\ap_n\in S$, then $\Pi\setminus S=\{\ap_{i_1},\dots,\ap_{i_s}\}$ with 
$i_1< \dots < i_s\le n-2$ and the consecutive diagonal blocks of $\el(S)$ are
$\mathfrak{gl}_{i_1}$, $\mathfrak{gl}_{i_2-i_1}$,\dots, $\mathfrak{gl}_{i_s-i_{s-1}}$, 
$\mathfrak{so}_{2(n-i_s)}$, and then the same $\mathfrak{gl}$-blocks in the reverse order.
Then $\un{a}:=(i_1,i_2-i_1,\dots,i_s-i_{s-1})$ and hence
$|\un{a}|=i_s$.
\\ \indent
(2ii) \ If $\ap_n\not\in S$, then $\Pi\setminus S=\{\ap_{j_1},\dots,\ap_{j_s}, \ap_n\}$ with
$j_s\le n-1$. Here the consecutive diagonal blocks  of $\el(S)$ are
$\mathfrak{gl}_{j_1}$, $\mathfrak{gl}_{j_2-j_1}$,\dots, $\mathfrak{gl}_{j_s-j_{s-1}}$, 
$\mathfrak{gl}_{n-j_s}$, and then the same $\mathfrak{gl}$-blocks in the reverse order.
Then $\un{a}=(j_1,j_2-j_1,\dots,j_s-j_{s-1}, n-j_s)$ with $|\un{a}|=n$.
\end{proof}

\begin{figure}[htb]
\begin{center}
\begin{tikzpicture}[scale= .6]
\draw (0,0)  rectangle (10,10);
\path[draw]  (4,0) -- (4,10); \path[draw]  (5,0) -- (5,10); \path[draw]  (6,0) -- (6,10);

\path[draw]  (0,4) -- (10,4); \path[draw]  (0,5) -- (10,5); \path[draw]  (0,6) -- (10,6);

\path[draw,fill=gray!20]  (0,6) -- (4,6) -- (4,10) -- (0,10)--cycle ;
\path[draw,fill=gray!20]  (6,0) -- (6,4) -- (10,4) -- (10,0)--cycle ;
\path[draw,fill=gray!20]  (4,0) -- (5,0) -- (5,4) -- (4,4)--cycle ;
\path[draw,fill=gray!20]  (5,6) -- (6,6) -- (6,10) -- (5,10)--cycle ;
\path[draw,fill=gray!20]  (0,4) -- (4,4) -- (4,5) -- (0,5)--cycle ;
\path[draw,fill=gray!20]  (6,5) -- (10,5) -- (10,6) -- (6,6)--cycle ;
\path[draw,fill=gray!20]  (4,5) -- (5,5) -- (5,6) -- (4,6)--cycle ;
\path[draw,fill=gray!20]  (5,4) -- (5,5) -- (6,5) -- (6,4)--cycle ;
\draw (0.4,9.6)  node {\footnotesize {\color{darkblue}$\esi_1$}} ;
\draw (3.4,6.3)  node {\footnotesize {\color{darkblue}$\esi_{n{-}1}$}} ;
\draw (4.5,5.5)  node {\footnotesize {\color{darkblue}$\esi_n$}} ;
\draw (5.5,4.5)  node {\footnotesize {\color{darkblue}$-\esi_n$}} ;
\draw (6.6,3.6)  node {\footnotesize {\color{darkblue}$-\esi_{n{-}1}$}} ;
\draw (9.5,0.4)  node {\footnotesize {\color{darkblue}$-\esi_1$}} ;
\draw[dashed,darkblue]  (0.7,9.3) -- (3.4,6.6) ;
\draw[dashed,darkblue]  (6.7,3.3) -- (9.4,0.6) ;
\draw[dashed,magenta]  (0,0) -- (10,10) ;
\end{tikzpicture}
\caption{The standard Levi $\el(S)\subset\sone$ with $\Pi\setminus S=\{\ap_{n-1}\}$} 
\label{fig:non-block-D}
\end{center}
\end{figure}
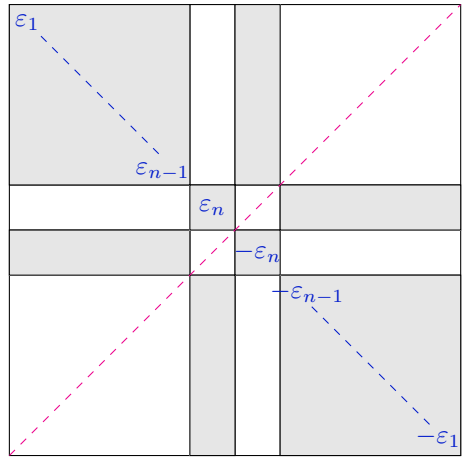

\begin{ex}
{\sf (i)} \ For the fixed Borel $\be$, Proposition~\ref{prop:D-parab-and-comp} yields $\un{a}=(1,\dots,1)=:(1^n)$; 
\\
{\sf (ii)} \ if $S=\Pi\setminus\{\ap_n\}$, then $\un{a}=(n)$;
\\
{\sf (iii)} \ 
$\p=\sone$ corresponds to the empty composition `$\varnothing$' with sum $0$.
 \end{ex}

\begin{df}   \label{df:admissible}
A subset $S\subset\Pi$ and the parabolics $\p(S), \p^-(S)\subset\sone$ are said to be {\it admissible}, 
if \ref{prop:D-parab-and-comp}(1) does not hold, i.e., either $\ap_{n-1},\ap_n\in S$ or $\ap_n\not\in S$.
\end{df}

By Proposition~\ref{prop:D-parab-and-comp}, to any standard (or opposite standard) admissible parabolic 
$\p$ in $\sone$ one naturally associates the composition $\un{a}$ with  $|\un{a}|\le n$ and 
$|\un{a}|\ne n-1$. 
There is a sort of 
inverse procedure that associates a standard admissible parabolic in $\sone$ to {\bf any} composition 
$\un{a}$ with  $|\un{a}|\le n$. Given $\un{a}=(a_1,\dots,a_s)$, we define the ``symmetric'' composition of 
$2n$ by $\un{\tilde a}:=(\un{a}, 2d, \un{a}^{-1})$, where $d=n-|\un{a}|$. Let $\p^{\sf A}(\un{\tilde a})$ be the 
standard ``symmetric'' parabolic in $\gltn$.  Then we associate to $\un{a}$ the admissible parabolic 
$\p(\un{a}):=\p^{\sf A}(\un{\tilde a})\cap \sone$. The standard Levi in $\el(\un{a})\subset \p(\un{a})$ has the 
consecutive diagonal blocks 
$\gt{gl}_{a_1}$, $\gt{gl}_{a_2}$,\dots, $\gt{gl}_{a_s}$, 
$\gt{so}_{2d}$, and then the same $\gt{gl}$--blocks in the reverse order. 
Hence, for $|\un{a}|\ne n-1$, we get the inverse map to one constructed in 
Proposition~\ref{prop:D-parab-and-comp}(ii).

\begin{rmk}   \label{rmk:exclude}
Since $\gt{so}_{2}\simeq\gt{gl}_1$, the compositions $\un{a}'$ with $|\un{a}'|=n-1$  and 
$\un{a}=(\un{a}',1)$ determine one and the same parabolic in $\sone$. For, $\gt{so}_2$ 
appearing as the middle block of $\el(\un{a}')$ is also the last $\gt{gl}_1$ contained in 
$\el(\un{a})\subset \gln=\el(\Pi\setminus\{\ap_n\})\subset \sone$. That is, some admissible $S\subset\Pi$ 
give rise to two standard symmetric parabolics $\tilde\p, \tilde\p'$ in $\gltn$ such that 
$\p(S)=\tilde\p\cap\sone=\tilde\p'\cap\sone$.
More precisely, this happens if and only if neither $\ap_{n-1}$ nor $\ap_n$ belongs to $S$.
\\ \indent
For this reason, we {\bf exclude} the compositions of $n-1$ from the further consideration.
\end{rmk}

\begin{df}     \label{df:crossing}
Let us say that $\q(S,T)=\p(S)\cap\p^-(T)$ is a seaweed {\it with crossing} (=\,has a crossing), if $\alpha_{n-1}\in T\setminus S$ and 
$\alpha_{n}\in S\setminus T$ (or vice versa). In the other cases, $\q(S,T)$ is said to be a seaweed 
{\it without crossing} (=\,has no crossing).
\end{df}

The full  meaning of these terms will be clarified below when we introduce the meander graphs for seaweeds with or without crossing.

\subsection{Seaweeds without crossing, compositions,  and meander graphs} 
\label{subs:without}

\begin{prop}    \label{prop:without}
Suppose that\/ $\q(S,T)$ has no crossing. 
\\ \indent
{\sf (i)} \ Then, up to permutation of $\ap_{n-1}$ and $\ap_n$, we may assume that both $S$ and $T$ are 
admissible and $\p(S)=\p(\un{a})$, $\p^-(T)=\p^-(\un{b})$.
In particular, $\q(S,T)$ has a ``seaweed shape''.
\\ \indent
{\sf (ii)} \ If\/ $S\cup T\supset \{\ap_{n-1},\ap_n\}$,  
then we may assume that $|\un{a}|\ne n-1$ and $|\un{b}|\le n-2$.
\\ \indent
{\sf (iii)} \ If\/ $\ap_{n-1}$ or $\ap_n$ does not belong to $S\cup T$, then $\q(S,T)$ lies in a Levi isomorphic to $\gln$. Here $\q(S,T)$ is given by two compositions with $|\un{a}|=|\un{b}|=n$.
\end{prop}
\begin{proof}
(i) \ If  at least one of the subsets $S, T$ is not admissible, then swapping  $\ap_{n-1}$ and $\ap_n$
makes both of them admissible, since $\q(S,T)$ has no crossing. Then  
$\un{a}$ and $\un{b}$ can independently be constructed as in Prop.~\ref{prop:D-parab-and-comp}.

(ii) \  Since $\q(S,T)$ has no crossing, we may assume w.l.o.g. that $\ap_{n-1},\ap_n\in T$, hence 
$\Pi\setminus T=\{\ap_{j_1},\dots,\ap_{j_t}\}$ and $j_t\le n-2$. Here $\p^-(T)=\p^-(\un{b})$ with 
$\un{b}=(j_1,j_2-j_1,\dots,j_t-j_{t-1})$ and
$|\un{b}|=j_t$, see Prop.~\ref{prop:D-parab-and-comp}(2i).
Then there are three possibilities for $S$:

{\it\bfseries (1)} \ If $\ap_{n-1},\ap_n\in S$, then we construct the composition $\un{a}$ for 
$\p(S)$ by the same rule. Here $|\un{a}|\le n-2$ as well.

{\it\bfseries (2)} \ If  $\ap_n\not\in S$, then $\Pi\setminus S=\{\ap_{i_1},\dots,\ap_{i_s}, \ap_n\}$ with
$i_s\le n-1$. Here the corresponding composition is $\un{a}=(i_1,i_2-i_1,\dots,i_s-i_{s-1}, n-i_s)$ with $|\un{a}|=n$, see Prop.~\ref{prop:D-parab-and-comp}(2ii).

{\it\bfseries (3)} \  If $\ap_{n-1}\not \in S$ and $\ap_n\in S$, then $S$ is not admissible.
But there is no harm in swapping $\ap_{n-1}$ and $\ap_n$. This does not 
change $\p^-(T)$ and yields an isomorphic seaweed, as in {\it\bfseries (2)}.

(iii) \ Regarding $S,T$ as subsets of the set of simple roots of $\gln=\el(\Pi\setminus\{\ap_{n-1}\})$ or
$\el(\Pi\setminus\{\ap_{n}\})$, we construct the required compositions of $n$ as explained in Section~\ref{subs:A}.
\end{proof}

A seaweed without crossing $\q(S,T)\subset\sone$ is also denoted by $\q_n(\un{a}|\un{b})$, where 
$\un{a}=\un{a}(S)$ and $\un{b}=\un{b}(T)$ are the associated compositions (constructed in 
Proposition~\ref{prop:without}) such that $|\un{a}|\ne n-1$ and $|\un{b}|\ne n-1$. Given $\un{a}$ and 
$\un{b}$, we form the symmetric compositions $\un{\tilde a}$ and $\un{\tilde b}$ of $2n$, as above.
Let $\p^{\sf A}(\un{\tilde a})$ and $\p^{{\sf A},-}(\un{\tilde b})$ be the corresponding standard 
"symmetric" parabolics in 
$\gltn$, $\q^{\sf A}(\un{\tilde a}|\un{\tilde b})=\p^{\sf A}(\un{\tilde a})\cap\p^{{\sf A},-}(\un{\tilde b})$ the  
standard seaweed in $\gltn$, and  $\Gamma^{\sf A}(\un{\tilde a}| \un{\tilde b})$ the corresponding 
type-{\sf A} meander graph.

\begin{df}    \label{def:graf-bez}
Let $\q=\q(S,T)\subset \sone$ be a seaweed {\bf without} crossing.  If 
$\un{a}=\un{a}(S)$ and $\un{b}=\un{b}(T)$ are the associated compositions, 
then the {\it type-{\sf D} meander graph\/} of $\q$ is
\[
 \Gamma_n(\q)=\Gamma_n(\un{a}|\un{b}):=\Gamma^{\sf A}(\un{\tilde a}| \un{\tilde b})
\] 
We also write $\Gamma_n(S,T)$ for this graph. (Note that $|\un{a}|\ne n-1$ and $|\un{b}|\ne n-1$.)
\end{df}
\begin{rmk}
Because different arcs in the type-{\sf A} meander graphs do not cross each other, the same holds for
the type-{\sf D} meander graph of a seaweed without crossing.
\end{rmk}

\begin{ex}   \label{ex:ala-Kos}
Suppose that $|\un{a}|=n$,  $|\un{b}|=n$, and $\q=\q_n(\un{a}|\un{b})$. Then 
$\q\subset\gln\subset\sone$, where $\gln=\el(\Pi\setminus\{\ap_n\})$. Here $\Gamma_n(\un{a}| \un{b})$ 
consists of two disjoint halves that are symmetric w.r.t. the middle. The first (resp. second) half represents 
the meander graph of the seaweed  $\q^{\sf A}(\un{a}|\un{b})$ (resp. $\q^{\sf A}(\un{a}^{-1}|\un{b}^{-1})$) 
in $\gln$. For instance, if $\un{a}=(2,2,1), \un{b}=(3,2)$, and $n=5$, then $S=\{\ap_1,\ap_3\}$, 
$T=\{\ap_1,\ap_2,\ap_4\}$ and the meander graph for $\q(S,T)$ is depicted in Fig.~\ref{pikcha:D5-kos}.
\begin{figure}[htb]
\setlength{\unitlength}{0.025in}
\begin{center}
\begin{picture}(115,18)(5,-2)
\put(-25,1){\small $\Gamma_5(2,2,1| 3,2)$:}
\multiput(20,3)(10,0){10}{\circle*{2}}

{\color{blue}
\put(30,5){\oval(20,10)[t]}
\put(55,5){\oval(10,5)[t]}
\put(75,5){\oval(10,5)[t]}
\put(100,5){\oval(20,10)[t]}

\put(25,1){\oval(10,5)[b]}
\put(45,1){\oval(10,5)[b]}
\put(85,1){\oval(10,5)[b]}
\put(105,1){\oval(10,5)[b]}
}
\qbezier[22](65,-8),(65,0),(65,12)
\end{picture}
\end{center}
\caption{}
\label{pikcha:D5-kos}
\end{figure}
\end{ex}
\begin{ex}   \label{ex:graph-b}
By Proposition~\ref{prop:without}, the associated compositions for the Borel 
$\be=\q(\varnothing,\Pi)$ are $\un{a}=(1^n)$ and $\un{b}=\varnothing$. 
Here $\tilde{\un{a}}=(1^{2n})$ and $\tilde{\un{b}}=(2n)$.
Therefore, $\Gamma_n(\be)$ has $n$ embedded arcs over the horizontal line and no arcs  
below the horizontal line, see Fig.~\ref{pikcha:Borel} with $n=5$.

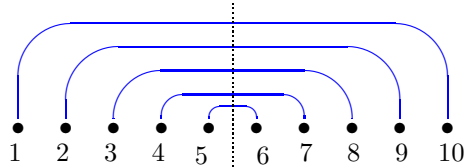
\begin{figure}[htb]
\setlength{\unitlength}{0.025in}
\begin{center}
\begin{picture}(115,26)(0,0)

\multiput(20,3)(10,0){10}{\circle*{2}}

{\color{blue}
\put(65,5){\oval(90,40)[t]}
\put(65,5){\oval(70,30)[t]}
\put(65,5){\oval(50,20)[t]}
\put(65,5){\oval(30,10)[t]}
\put(65,5){\oval(10,5)[t]}
}
\qbezier[35](65,-5),(65,12),(65,29)
\put(18,-4){\footnotesize $1$}
\put(28,-4){\footnotesize $2$}
\put(38,-4){\footnotesize $3$}
\put(48,-4){\footnotesize $4$}
\put(57,-4.5){\footnotesize $5$}
\put(70,-4.5){\footnotesize $6$}
\put(79,-4){\footnotesize $7$}
\put(89,-4){\footnotesize $8$}
\put(99,-4){\footnotesize $9$}
\put(108,-4){\footnotesize $10$}
\end{picture}
\end{center}
\caption{The type-{\sf D} meander graph for  $\be\subset\gt{so}_{10}$}   \label{pikcha:Borel}
\end{figure}
\end{ex}

The type-{\sf D} meander graphs are symmetric w.r.t. the vertical line between the $n$-th and $(n+1)$-th 
vertices (this holds for the seaweeds with or without crossing). The symmetry w.r.t. this line is denoted by 
$\sigma$ and this line is said to be the $\sigma$-{\it mirror}. This line is depicted by the dotted line in the 
figures.

\subsection{Seaweeds with crossing and their meander graphs}   \label{subs:with}
Quite a different situation occurs if $\q$ has a crossing. 
The three steps of our definition/construction of $\Gamma_n(\q)$ are:

{\sf 1)} If $\q(S,T)$ has a crossing and $S$ is not admissible, then $S$ is replaced with $\check{S}$, so that 
$\check{\q}=\q(\check{S},T)$ has no crossing.

{\sf 2)} Following 
Definition~\ref{def:graf-bez}, we construct the meander graph $\Gamma_n(\check{S},T)$ .

{\sf 3)} We make a certain alteration in $\Gamma_n(\check{S},T)$, and the resulting graph is defined 
to be the meander graph of $\q(S,T)$.

For {\sf 1)}: \ Let $\q=\q(S,T)$ be a seaweed with crossing and $\alpha_{n-1}\in T\setminus S$, 
$\alpha_{n}\in S\setminus T$. The admissible subset $T$ gives rise to a composition 
$\un{b}=(b_1,\dots,b_t)$, see Proposition~\ref{prop:D-parab-and-comp}. As $S$ is not admissible, 
we replace $\ap_{n}$ with $\ap_{n-1}$ in it. This yields an admissible subset 
$\check{S}$ and the corresponding composition $\un{a}=(a_1,\dots,a_s)$. 
(We do not change $T$!)
The structure of $\check{S}$ and $T$ shows that $|\un{a}|=|\un{b}|=n$, 
$a_s\ge 2$, and $b_t\ge 2$, cf. the proof of Proposition~\ref{prop:D-parab-and-comp}(2ii).
Note that $\ap_n\not\in\check{S}\cup T$, hence $\check{\q}=\q(\check{S},T)$ lies in the Levi 
$\el(\Pi\setminus\{\ap_n\})\simeq \gln$.
\\ \indent
For {\sf 2)}: \ Since $\check{\q}$ has no crossing, we obtain the meander graph 
$\Gamma_n(\check{\q})=\Gamma_n(\un{a},\un{b})$. It consists of two symmetric copies of the type-{\sf A} 
meander graphs for seaweeds in $\gln$ (cf. Example~\ref{ex:ala-Kos} and Fig.~\ref{pikcha:D5-kos}). Recall 
that the arcs below (resp{.}~over) the horizontal line are determined by a symmetric composition of $2n$, 
which in our case is $(\un{a},\un{a}^{-1})$ (resp{.}~$(\un{b},\un{b}^{-1})$). But this is not $\Gamma_n(\q)$
yet. As $S$ has been changed, we have to reflect this in the graph. 
\\ \indent
For {\sf 3)}: \ 
If $a_s < b_t$, then we modify two largest arcs \un{below} the horizontal line that correspond to the two 
parts $a_s$ in the middle of $(\un{a},\un{a}^{-1})$. That is, the arc from $n-a_s+1$ will go not to $n$, but to 
$n+1$; and the arc from $n+a_s$ goes now to $n$ in place of $n+1$. If $a_s> b_t$, then the same 
procedure applies to the both parts $b_t$ and two arcs \un{over} the horizontal line. 
If $a_s=b_t$, then either of the sides is suitable for alteration, because the two resulting graphs are isomorphic.
This alteration yields two arcs crossing each other, which explains the term "crossing".
A sample case is depicted below, where $a_s=4\le b_t$ and we do not draw the other arcs, over or below the horizontal line.
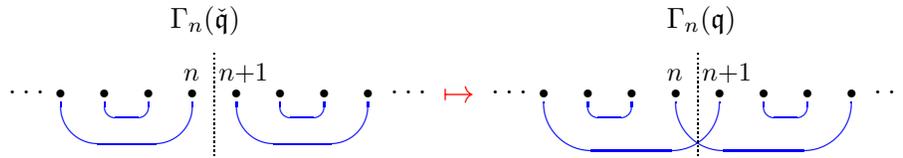
\begin{figure}[htb]
\setlength{\unitlength}{0.023in}
\begin{center}
\begin{picture}(200,23)(-5,-4)
\put(-12,3){$\dots$}
\put(75,3){$\dots$}
\put(28,6){\footnotesize $n$}
\put(36,6){\footnotesize $n{+}1$}
\multiput(0,3)(10,0){8}{\circle*{2}}
{\color{blue}
\put(15,1){\oval(10,7)[b]}
\put(15,1){\oval(30,19)[b]}
\put(55,1){\oval(10,7)[b]}
\put(55,1){\oval(30,19)[b]}
}
\qbezier[22](35,-11),(35,0),(35,12)
{\color{red}\put(87,1){$\mapsto$}
}
\put(25,18){\small $\Gamma_n(\check{\q})$}
\put(98,3){$\dots$}
\put(185,3){$\dots$}
\put(138,6){\footnotesize $n$}
\put(146,6){\footnotesize $n{+}1$}
\multiput(110,3)(10,0){8}{\circle*{2}}
{\color{blue}
\put(125,1){\oval(10,7)[b]}
\put(130,1){\oval(40,22)[b]}
\put(165,1){\oval(10,7)[b]}
\put(160,1){\oval(40,22)[b]}
}
\put(138,18){\small $\Gamma_n(\q)$}
\qbezier[22](145,-11),(145,0),(145,12)
\end{picture}
\end{center}
\end{figure}

\noindent
The graph obtained is the desired type-{\sf D} meander graph of a seaweed with crossing. Our 
construction justifies the notation $\q(S,T)=\q_n(\un{a}|\un{b})_{\sf c}$ for
seaweeds with crossing. In this case, we also write $\Gamma_n(S,T)=
\Gamma_n(\un{a}|\un{b})_{\sf c}$. 

{\it \bfseries  Remarks.}  {\sf (1)}\ The above alteration shouldn't be regarded as the permutation of vertices $n$ 
and $n+1$, because the arcs on the other side of the horizontal line are not affected!
\\ \indent
{\sf (2)}\ The rule is that two arcs crossing each other are related to the smaller part among 
$\{a_s,b_t\}$. We then say that crossing is on the {\it correct side\/} of the meander graph.  Otherwise, the 
crossing is on the {\it wrong side}. If $a_s=b_t$, then alteration can be made on any side, i.e.,
both sides are correct. Although, we initially
consider the graphs with crossing on the correct side, it can happen that after some reduction steps we 
obtain a graph with crossing on the wrong side. In that case, we 
will need further adjustments, see Section~\ref{sect:main}.
\\ \indent
{\sf (3)}\  For seaweeds without crossing, the sum of the compositions $\un{a}$ and $\un{b}$ is not fixed. Therefore, we always put the index $n$ in the notation for $\q_n(\un{a}|\un{b})\subset\sone$. While for the seaweeds with crossing, the sum is always $n$. Hence the notation $\q(\un{a}|\un{b})_{\sf c}$ is unambiguous.  

\vskip.7ex 
As a by-product of the definition,  we have the following observation:

\noindent   {\it If\/ $\q$ has a crossing, then there are exactly two arcs that cross each other in 
$\Gamma_n(\q)$. These two arcs are also the only arcs crossing the $\sigma$-mirror.
}

The connected components of $\Gamma_n(\q)$ through these two arcs are said to be {\it strange}.
It is easily seen that either 
these two arcs lie in the same connected component, which is a "strange" cycle, 
or they lie in two different ("strange") segments and $\sigma$ permutes these segments.

\begin{ex}   \label{ex:basic-crossing-q_ec}
The basic and most essential example of a seaweed with crossing is
$\q_{\sf ec}(n)=\q_{\sf ec}:=\q(\Pi\setminus\{\alpha_{n-1}\},\Pi\setminus\{\alpha_n\})\subset\sone$.

Here $\un{a}=\un{b}=(n)$ and the resulting meander graph for $n=5$ is depicted in 
Fig.~\ref{pikcha:D5-1}.

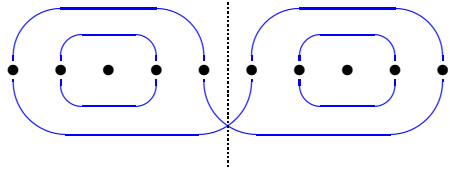
\begin{figure}[htb]
\setlength{\unitlength}{0.025in}
\begin{center}
\begin{picture}(140,25)(-5,-8)

\multiput(20,3)(10,0){10}{\circle*{2}}

{\color{blue}
\put(40,5){\oval(40,22)[t]}
\put(40,5){\oval(20,11)[t]}
\put(90,5){\oval(40,22)[t]}
\put(90,5){\oval(20,11)[t]}

\put(45,1){\oval(50,23)[b]}
\put(40,1){\oval(20,11)[b]}
\put(85,1){\oval(50,23)[b]}
\put(90,1){\oval(20,11)[b]}
}
\qbezier[37](65,-17),(65,0),(65,17)
\end{picture}
\end{center}
\caption{The meander graph of  $\q_{\sf ec}(5)\subset\gt{so}_{10}$}   \label{pikcha:D5-1}
\end{figure}
\noindent
The graph $\Gamma_n(\q_{\sf ec})$ has a unique strange connected component (cycle).
\end{ex}

\begin{rmk}     \label{rmk:3-kuska}
Given $\q=\q_n(\un{a}|\un{b})$ or $\q_n(\un{a}|\un{b})_{\sf c}$, suppose that 
$\sum_{i=1}^k a_i=\sum_{j=1}^l b_j=m$ for some $m\le n-2$ or $m=n$.  Then 
$\q\subset \el(\Pi\setminus \{\ap_m\})\simeq \gt{gl}_m\oplus \gt{so}_{2(n-m)}$ and $\Gamma_n(\q)$ 
consists of three disjoint graphs. The central graph represents a seaweed in $\gt{so}_{2(n-m)}$ and two 
extreme symmetric graphs represent seaweeds in $\gt{gl}_m$. More precisely, if 
$\un{a}'=(a_1,\dots,a_k)$, $\un{b}'=(b_1,\dots,b_l)$, $\un{a}''=(a_{k+1},\dots,a_s)$, and 
$\un{b}''=(b_{l+1},\dots,b_t)$, then the central graph is either $\Gamma_{n-m}(\un{a}''|\un{b}'')$ 
or $\Gamma_{n-m}(\un{a}''|\un{b}'')_{\sf c}$; and two other graphs are $\Gamma^{\sf A}(\un{a}'|\un{b}')$ and
$\Gamma^{\sf A}(\un{a}'^{-1}|\un{b}'^{-1})$, see Fig.~\ref{pikcha:tri-kuska}.

\begin{figure}[htb]
\setlength{\unitlength}{0.025in}
\begin{center}
\begin{picture}(150,30)(-10,-8)

\put(-1,3){$\Gamma^{\sf A}(\un{a}'|\un{b}')$}
\put(44,12){\small $\Gamma_{n-m}(\un{a}''|\un{b}'')$}
\put(57,4){\footnotesize or}
\put(43,-3){\small $\Gamma_{n-m}(\un{a}''|\un{b}'')_{\sf c}$}
\put(93,3){$\Gamma^{\sf A}(\un{a}'^{-1}|\un{b}'^{-1})$}

\thicklines
{\color{redi}
\put(10,5){\oval(40,30)}
\put(60,5){\oval(40,30)}
\put(110,5){\oval(40,30)}
}
\qbezier[25](60,-17),(60,4),(60,25)
\end{picture}
\end{center}
\caption{ }
\label{pikcha:tri-kuska}
\end{figure}
If $m=n$, then the central graph disappears, cf. Example~\ref{ex:ala-Kos}.
\end{rmk}

\section{The index of seaweeds via type-{\sf D} meander graphs}
\label{sect:main}

\noindent
By our constructions in Sections~\ref{subs:without} and \ref{subs:with}, each connected component of a 
type-{\sf D} meander graph is homeomorphic to either a cycle or segment. An isolated vertex is regarded 
as a segment. For instance, there are two segments and three cycles in Fig.~\ref{pikcha:D5-1}. The arcs 
crossing the $\sigma$-mirror are said to be {\it central}.  Our main result is the following formula for the 
index of a standard seaweed $\q$ in terms of the connected components of $\Gamma_n(\q)$. 

\begin{thm}    \label{thm:main}
Let $\q\subset\sone$ be a standard seaweed and\/ $\Gamma=\Gamma_n(\q)$  the corresponding 
type-{\sf D} meander graph.  Then
\beq  \label{eq:index-D}
    \ind\q=\#\{\text{the cycles of }\ \Gamma\} + 
   \frac{1}{2}\#\{\text{the non-$\sigma$-stable segments of\/ $\Gamma$}\} +\epsilon , 
\eeq
where $\epsilon=\epsilon(\q) 
\in \{0,\pm 1\}$ is determined by the following rules. 
\begin{itemize}
\item[$\langle\diamond_1\rangle$]
Suppose that $\q$ has no crossing, 
$\q=\q_n(\un{a}|\un{b})$, and $\Gamma_n(\q)=\Gamma_n(\un{a}|\un{b})$. 
Let $m_a$ and $m_b$ be the number of central arcs 
below and above the horizontal line, respectively. Assuming that $m_b\ge m_a$, we set
\begin{itemize}
\item $\epsilon=0$ if  $m_b-m_a$ is even; 
\item $\epsilon=1$ if $m_b$ is odd, $m_a=0$, and the arc between $n$ and $n+1$ belongs to a segment;   
\item $\epsilon=-1$ in the remaining cases (with $m_b-m_a$ odd). 
\end{itemize}
\item[$\langle\diamond_2\rangle$] If\/ $\q$ has a crossing, then there are two possibilities:
\begin{itemize}
\item if\/ $\Gamma_n(\q)$ has a unique \emph{strange} component (cycle), 
then $\epsilon=-1$; 
\item if there are two \emph{strange} segments (= the segments crossing the $\sigma$-mirror), then $\epsilon=0$. 
\end{itemize}
\end{itemize}
\end{thm}

\begin{ex}   \label{ex:two-cases}
1) \ The first possibility in $\langle\diamond_2\rangle$ realises for $\q_{\sf ec}(n)$, 
see~Fig.~\ref{pikcha:D5-1} for $n=5$. Hence $\ind \q_{\sf ec}(5)=3$. The second possibility occurs for 
$\q=\q(S,T)\subset \mathfrak{so}_{10}$ with $S=\{\ap_1,\ap_3,\ap_5\}$ and
$T=\{\ap_2,\ap_3,\ap_4\}$. Then $\un{a}=(2,3)$, $\un{b}=(1,4)$, and $\Gamma_5(\q)=
\Gamma_5(2,3|1,4)_{\sf c}$, see~Fig.~\ref{pikcha:D5-2}. 
Here $\Gamma_5(\q)$ has two 
strange segments. Therefore, $\epsilon=0$ and $\ind \q=1$.
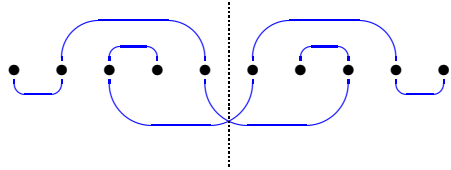
\begin{figure}[htb]
\setlength{\unitlength}{0.025in}
\begin{center}
\begin{picture}(140,25)(-5,-8)

\multiput(20,3)(10,0){10}{\circle*{2}}

{\color{blue}
\put(45,5){\oval(10,6)[t]}
\put(45,5){\oval(30,17)[t]}
\put(85,5){\oval(10,6)[t]}
\put(85,5){\oval(30,17)[t]}

\put(55,1){\oval(30,19)[b]}
\put(25,1){\oval(10,6)[b]}
\put(75,1){\oval(30,19)[b]}
\put(105,1){\oval(10,6)[b]}
}
\qbezier[37](65,-17),(65,0),(65,17)
\end{picture}
\end{center}
\caption{The meander graph $\Gamma_5(2,3|1,4)_{\sf c}$}   \label{pikcha:D5-2}
\end{figure}

2) \ For $\q=\q_5(2,2,1|3,2)\subset\gt{so}_{10}$ (Fig.~\ref{pikcha:D5-kos}), we have $\epsilon=0$
and $\ind\q=\ind\q^{\sf A}(2,2,1|3,2)=1$.

3) \ We have $\epsilon=1$ for $\be\subset \gt{so}_{10}$ (Fig.~\ref{pikcha:Borel}). Hence $\ind\be=\epsilon=1$. 
\end{ex}
\noindent
In our proof of Theorem~\ref{thm:main}, we use  
the inductive procedure of~\cite{Dima01}. That procedure allows us to reduce computation of the index 
of arbitrary seaweeds to the case of either a 
parabolic subalgebra in $\gt{so}_{2m}$ or the seaweed with crossing $\q_{\sf ec}(m)$ for some $m\le n$ 
(see below). For this reason, we begin with the case of parabolics and seaweeds $\q_{\sf ec}$. 
In dealing with the parabolics, the general {\it Tauvel--Yu--Joseph formula} (=\,{\it TYJ formula}) for the 
index of a seaweed $\q(S,T)\subset \g$ is required. Let $\cK(\el(S))=:\cK(S)$ be the cascade of strongly 
orthogonal roots (=\,{\it Kostant's cascade}) in the Levi subalgebra $\el(S)$, see \cite{jos77,ty-AIF} for the 
details. In particular, $\cK(\Pi)=\cK(\g)$ is the cascade in the whole of $\g$.
Let $E_S$ be the linear span of $\cK(S)$ in
$\te^*_{\BR}$. Then $\dim E(S)=\# \cK(S)$ and the TYJ formula reads: 
\beq    \label{eq:tau-yu}
   \ind\q(S,T)=\rk \g +\dim E_S +\dim E_T  -2\dim (E_S+E_T) ,
\eeq
see~\cite[Conj.\,4.7]{ty-AIF} and \cite[Section\,8]{jos}. Clearly, $\cK(\el_1\oplus\el_2)=\cK(\el_1)\sqcup\cK(\el_2)$. For future use, we record the data
on the cascade in $\gln$ and $\sone$. For $\g=\gln$, we have
\beq  \label{eq:K-Pi-gl}
       \cK(\Pi_{\gln})=\{\esi_i-\esi_{n+1-i}\mid i=1,\dots, [n/2]\} \text{ and } \ \# \cK(\Pi_{\gln})=[n/2].
\eeq
For $\g=\sone$, we have
$\cK(\Pi_{\sone})=\{\esi_1\pm\esi_2,\dots,\esi_{2l-1}\pm\esi_{2l}\}$ if $n=2l,2l+1$. Therefore, 
\beq   \label{eq:K-Pi}
   \# \cK(\Pi_{\sone})=\begin{cases} \rk\g=n, & \text{ if $n$ is even}, \\
      \rk\g-1=n-1,  &\text{ if $n$ is odd}.  \end{cases} 
\eeq

\begin{lm}        \label{Ind-base}
Formula~\eqref{eq:index-D} holds for all parabolic subalgebras and the seaweed $\q_{\sf ec}$ in $\sone$.
\end{lm}
\begin{proof}
1)  Using the explicit matrix model of $\q_{\sf ec}$, we notice that it is isomorphic to the semi-direct product $(\gt{gl}_{n-1}{\oplus}\gt{gl}_{1})\ltimes(\BC^{n-1}{\oplus}(\BC^{n-1})^*)$, where $\BC^{n-1}$ and $(\BC^{n-1})^*$ are standard dual $\gt{gl}_{n-1}$-modules and 
the weights of the $2$-dim centre of $\gt{gl}_{n-1}{\oplus}\gt{gl}_{1}$ on $\BC^{n-1}$ and $(\BC^{n-1})^*$ are linearly independent, see the picture.
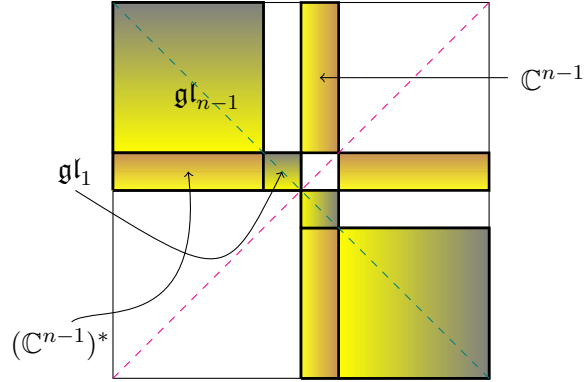
\begin{center}
\begin{tikzpicture}[scale= .5]
\draw (0,0)  rectangle (10,10);

\path[draw,line width=1pt]  (0,5) -- (5,5) -- (5,6) -- (4,6) -- (4,10) -- (0,10) -- (0,5) ;

\shade[bottom color=yellow,top color=gray,draw, line width=1pt]  (0,6) -- (4,6) -- (4,10) -- (0,10) -- (0,6) ;

\path[draw,line width=1pt]  (5,0) -- (5,5) -- (6,5) -- (6,4) -- (10,4) -- (10,0) -- (5,0); 

\shade[left color=yellow,right color=gray, draw,line width=1pt]  (6,0) -- (6,4) -- (10,4) -- (10,0) -- (6,0); 

\shade[bottom color=yellow,top color=brown,draw,line width=1pt]  (0,5) -- (4,5) -- (4,6) -- (0,6) -- (0,5) ;

\shade[left color=yellow,right color=brown,draw,line width=1pt]  (5,0) -- (5,4) -- (6,4) -- (6,0) -- (5,0) ;

\shade[left color=yellow,right color=brown,draw,line width=1pt]  (5,6) -- (6,6) -- (6,10) -- (5,10) -- (5,6) ;

\shade[top color=brown,bottom color=yellow,draw,line width=1pt]  (6,5) -- (6,6) -- (10,6) -- (10,5) -- (6,5) ;

\shade[bottom color=yellow,top color=gray,draw, line width=1pt] (4,5) -- (5,5) -- (5,6) -- (4,6) -- (4,5) ;

\shade[left color=yellow,right color=gray,draw, line width=1pt] (5,4) -- (5,5) -- (6,5) -- (6,4) -- (5,4) ;
\draw[->]   (-1, 5) .. controls (3,2.5) .. (4.5,5.5);

\draw[->]   (-1.1,1.5) .. controls (-0.5,2) and (2.5,1) .. (2,5.5);
\draw[->]   (10.5,8) .. controls (8,8)  .. (5.5,8);

\draw (2.5,7.5)  node {\small $\mathfrak{gl}_{n-1}$} ;
\draw (-1,5.5)  node {\small $\mathfrak{gl}_{1}$} ;
\draw (-1.4,.9)  node {\small $(\mathbb C^{n-1})^*$} ;
\draw (11.7,8)  node {\small $\mathbb C^{n-1}$} ;

\draw[dashed,magenta]  (0,0) -- (10,10) ;
\draw[dashed,my_color]  (10,0) -- (0,10) ;
\end{tikzpicture}
\end{center}

Applying the Ra\"{i}s formula for the index of semi-direct products~\cite{rais}, we then obtain
$\ind\q_{\sf ec}=n-2$. 
On the other hand, 
\\ \indent
-- if $n$ is even, then $\Gamma_n(\q_{\sf ec})$ consists of $n-1$ cycles;
\\ \indent
-- if $n$ is odd, then $\Gamma_n(\q_{\sf ec})$ consists of $n-2$ cycles and two isolated points
(segments), which are not $\sigma$-stable.
\\
According to $\langle\diamond_2\rangle$,  here $\epsilon=-1$, which yields the value $n-2$ in
Eq.~\eqref{eq:index-D} in both cases.  

Thus, Ra\"{i}s' formula and \eqref{eq:index-D} give one and the same value for $\ind\q_{\sf ec}(n)$. 

2)  Let $\q$ be a standard parabolic, that is, $\q=\p(S)=\q(S,\Pi)$. W.l.o.g., we may assume that $S$ is 
admissible and then take the associated compositions $\boa:=(a_1,\dots,a_s)$ and $\un{b}=\varnothing$. Set $\Gamma=\Gamma_n(\un{a}|\varnothing)$ and 
$k(\un{a}):=\left[\frac{a_1}{2} \right ]+\ldots+\left[ \frac{a_s}{2} \right ]+d$ , where $d=n-|\un{a}|$. 
It is  easily seen that
\\[.6ex] \centerline{
$k(\un{a})=\#\{\text{the cycles of }\ \Gamma \} + 
\frac{1}{2}{\cdot}\#\{\text{the non-$\sigma$-stable segments of $\Gamma$}\}$ .
}\\[.6ex]
Therefore, Theorem~\ref{thm:main} gives the value $k(\boa)+\epsilon$ for $\ind\p(S)$. (Actually, all
segments of $\Gamma$ are $\sigma$-stable, so $k(\boa)$ equals just the number 
of cycles.) 
To apply the TYJ formula~\eqref{eq:tau-yu}, one has to distinguish even and odd $n$. Since 
$\el(S)= \gt{gl}_{a_1}\oplus\dots \oplus\gt{gl}_{a_s}\oplus \gt{so}_{2d}$, Formulae~\eqref{eq:K-Pi-gl} and
\eqref{eq:K-Pi} show that the value of $\# \cK(S)$ depends on the parity of $d$ as well. Namely,
\beq    \label{eq:K(S)}
    \# \cK(S) =\begin{cases} 
    \left[ \frac{a_1}{2} \right ]+\ldots+\left[ \frac{a_s}{2} \right ]+d=k(\boa), & \text{ if $d$ is even}; \\
    \left[ \frac{a_1}{2} \right ]+\ldots+\left[ \frac{a_s}{2} \right ]+d-1=k(\boa)-1, & \text{ if $d$ is odd}.
    \end{cases} 
\eeq

{\it \bfseries Suppose that $n$ is even}. Here 
Formulae~\eqref{eq:tau-yu} and \eqref{eq:K-Pi} give us that $\ind\gt p(S)=\# \cK(S)$.  

\textbullet \ \ If $d$ is even, then $\# \cK(S)=k(\boa)$. 
On the other hand, $m_b=n$, $m_a=d$, and  $n-d$ is even; hence $\epsilon=0$.

\textbullet \ \ If $d$ is odd, then $\# {\mathcal K}(S)=\koa-1$. On the other hand, $m_b$ is even, 
$m_a>0$, and $n-d$ is odd; hence $\epsilon=-1$.
\\ \indent
Thus, both Theorem~\ref{thm:main} and Eq.~\eqref{eq:tau-yu} give the same value for $\ind\p(S)$.

\vskip.7ex
{\it \bfseries Suppose that $n$ is odd}. We first mention the case of $n=1$ and $\g=\q=\gt{so}_2$,
since it may occur as a step in our future reduction procedure.
Then $\ind\q=1$ and by Definition~\ref{def:graf-bez}, the  
meander graph of $\q$ is 
$\Gamma_1(1|1)$. 
For this graph, Theorem~\ref{thm:main} also gives value $1$.

Until the end of the proof, we assume that $n\ge 3$ is odd. 
Then $E_{\Pi}$ is a subspace of $\te^*_\BR$ of codimension $1$. More precisely,
$E_{\Pi}=(\esi_n)^\perp$. 
Therefore, $\dim (E_S + E_\Pi)=n$ if and only if $\cK(S)$ contains  a root of the form $\esi_j\pm\esi_n$ for 
some $j< n$. This is determined by the ``last'' factor of $\el(S)$, which us either $\gt{gl}_{a_s}$ (if $d=0$) 
or $\gt{so}_{2d}$. Then an easy analysis shows that
\beq    \label{eq:condition-n}
   \dim (E_S + E_\Pi)=n \ \Longleftrightarrow \ 
   d\ge 2 \ \text{ is even or }  \  
   d=0 \ \ \&  \ \ a_s>1. 
\eeq
There are three possibilities now, and each time we compare the values given by Eq.~\eqref{eq:index-D} and the TYJ formula.

\textbullet \ \ If $ \dim (E_S + E_\Pi)=n$, 
then the TYJ formula gives
\\[.6ex]
\centerline{ $\ind\p(S)=n +\#\cK(S)+ (n-1)-2n=\#\cK(S)-1=k(\boa)-1$.
}
On the other hand, $m_b=n$ is odd and $m_a=d$ is even. It is also easily seen that in both cases
($d\ge 2$ or $d=0$ and $a_s>1$), the arc between vertices $n$ and $n+1$ belongs to a cycle.
Therefore, $\epsilon=-1$.

\textbullet \ \ If  $d=0$ and $a_s=1$, then $E_S+E_\Pi=E_\Pi$ and the TYJ formula gives
\\[.6ex]
\centerline{
$\ind\p(S)=n +\#\cK(S)+ (n-1)-2(n-1)=\#\cK(S)+1=k(\boa)+1$.
}
On the other hand, $m_b=n$ is odd and $m_a=0$. The condition that $a_s=1$ also implies that 
the arc between vertices $n$ and $n+1$ belongs to a segment.
Therefore, $\epsilon=1$.

\textbullet \ \ If  $d$ is odd, then still $E_S+E_\Pi=E_\Pi$ and, taking into account Eq.~\eqref{eq:K(S)},
 the TYJ formula gives
\\[.6ex]
\centerline{
$\ind\p(S)=n +\#\cK(S)+ (n-1)-2(n-1)=\#\cK(S)+1=k(\boa)$.
}
On the other hand, both $m_b=n$ and $m_a=d$ are odd. Hence $\epsilon=0$.

Thus, it is verified in all cases that $\ind\p(S)=k(\boa)+\epsilon$. 
\end{proof} 

\begin{rmk}
Explicit formulae for the index of the parabolic subalgebras of $\sone$ are obtained in~\cite[Section\,4]{dvor}. They could have been used in place of
the TYJ formula in the proof of Lemma~\ref{Ind-base}.
\end{rmk}

Let us recall the {\it inductive procedure\/} for computing the index of seaweeds in the classical  Lie 
algebras introduced by the first author~\cite{Dima01}. The aim of that procedure is to reduce computation 
of the index of arbitrary seaweeds to parabolic subalgebras. It is a good time to confess that there is 
a gap concerning the case of $\sone$ in~\cite[Sect.\,5]{Dima01}. 
Not any seaweed in $\sone$ can be reduced to a parabolic. 
Strictly speaking, because seaweeds {\bf with} crossing are not considered in~\cite{Dima01}, the applicability
of the inductive procedure to them is questionable. However, as we shortly see, the procedure can be adjusted so that it works unless $\q$ is parabolic or $\q=\q_{\sf ec}(m)$.
That is, the correct statement is that any standard seaweed in $\sone$ can be reduced to either a 
parabolic or $\q_{\sf ec}(m)\subset\gt{so}_{2m}$ for some $m\le n$.  

Suppose that $\un{a}=(a_1,\dots,a_s)$ and 
$\un{b}=(b_1,\dots,b_t)$ are two compositions with $|\un{a}| \le n$, $|\un{b}| \le n$, $|\un{a}|\ne n-1$,
and $|\un{b}|\ne n-1$. Consider the standard seaweed without crossing 
$\q_n(\un{a}{\mid} \un{b}) \subset \gt{so}_{2n}$.

{\sf \bfseries Inductive procedure}:

{\sl \bfseries Step 1.} If either $\un{a}$ or $\un{b}$ is empty, then $\q_n(\un{a}{\mid} \un{b})$ is a
parabolic, and there is no reduction. 
 
{\sl\bfseries Step 2.}  Suppose that both $\un{a}$ and $\un{b}$ are non-empty.
By \cite[Theorem\,5.2]{Dima01}, $\ind \q_n(\un{a}{\mid} \un{b})$ can 
recursively be computed as follows:
\par
{\sf (i)} \ If $a_1=b_1$, then $\q_n(\un{a}{\mid} \un{b})\simeq \gt{gl}_{a_1}\oplus \q_{n-a_1}(a_2,\dots,a_s{\mid} b_2,\dots,b_t)=:\gt{gl}_{a_1}\oplus \q'$, hence
\par
\centerline{$
  \ind \q_n(\un{a}{\mid} \un{b})=a_1+\ind \q_{n-a_1}(a_2,\dots,a_s{\mid} b_2,\dots,b_t)=a_1+\ind\q'$.}
\par
{\sf (ii)} \ If $a_1\ne b_1$, then $\ind\q_n(\un{a}{\mid}\un{b})=\ind\q'$, where $\q'$ is defined as follows. If $a_1< b_1$, then 
\beq    \label{q-shtrikh}
\q'=   
\left\{\begin{array}{ll} 
\q_{n-a_1}(a_2,\dots,a_s{\mid} b_1-2a_1,a_1,b_2,\dots,b_t) 
& \ {\mathrm{for\ }} a_1\le b_1/2 ; \\
\q_{n-b_1+a_1}(2a_1-b_1,a_2,\dots,a_s{\mid} a_1,b_2,\dots,b_t) 
& \ {\mathrm{for\ }} a_1 \ge b_1/2  ,
\end{array}\right.
\eeq
and likewise, if $a_1>b_1$.
\par{\sf (iii)} \ This step terminates if one of the compositions becomes empty, i.e., 
we obtain a parabolic subalgebra 
in a smaller orthogonal Lie algebra $\gt{so}_{2m}$.
\\[.7ex]
This procedure works also for types {\sf A,B,C}. In particular, if $\q^{\sf A}(\un{a}|\un{b})\subset\gln$
(and hence $|\un{a}|=|\un{b}|=n$), then the similar steps and formulae apply, see \cite[Theorem\,4.2]{Dima01}.

\begin{rmk}   \label{rem: preserve-difference}
The formulae of {\sl\bfseries Step~2} preserve the differences $n-|\un{a}|$ and $n-|\un{b}|$. 
For instance, if $a_1 >b_1/2$, then $n-|\un{a}|=(n-b_1+a_1)- \bigl((2a_1-b_1)+
\sum_{j=2}^s a_j\bigr)$. This means that the forbidden (excluded) compositions cannot occur after a reduction step, i.e., the {\sf \bfseries inductive procedure} is well-defined. (Recall that we exclude the compositions $\un{a}$ such that $n-|\un{a}|=1$.)
\end{rmk}

Let us explain how this procedure works if $\q=\q(S,T)$ has a crossing and, say, $S$ is not admissible.
By Section~\ref{subs:with}, we associate two compositions of $n$ with $\q$, 
$\un{a}=\un{a}(\check S)=(a_1,\dots,a_s)$ and $\un{b}=\un{b}(T)=(b_1,\dots,b_t)$, such that 
$\q=\q(\un{a}|\un{b})_{\sf c}$. We may as well assume that $a_s\le b_t$. The presence of crossing is 
expressed via the modification of the largest arcs associated with part $a_s$ in 
$\Gamma_n(\un{a}|\un{b})=\Gamma^{\sf A}(\un{a},\un{a}^{-1}|\un{b},\un{b}^{-1})$. In the situation of 
{\sl\bfseries Step 2}{(\sf i}), where $a_1=b_1$, we have $\q'=\q(a_2,\ldots,a_s | b_2,\ldots,b_t)_{\sf c}$.  
The formulae of {\sl\bfseries Step~2}{(\sf ii}) reflect certain invariant-theoretic manipulations with $\q$ 
that affect only the upper-left block $\gt{gl}_m\subset \q$, where $m=\max\{a_1,b_1\}$, see~\cite{Dima01}. 
Actually, $\q'$ is the stabiliser of a suitable $\xi\in\q^*$.  
The description of $\q'$ is independent of 
the parts $a_2,\dots,a_s; b_2,\dots, b_t$. Therefore, 
as long as  the part $a_s$ of a seaweed $\gt q$ with crossing 
is not involved in the reduction, we can pass to $\q'$ with $\ind\q'=\ind\q$. 
Mostly $\q'$ would be a seaweed with crossing defined by Eq.~\eqref{q-shtrikh} with the subscript 
`{\sf c}' in the RHS. But there are some exceptional cases, and this is to be clarified in the proof of Theorem~\ref{thm:main}. 

\begin{rmk}   \label{rmk:auch-for-graphs}
The procedure can be thought of as one that applies to the triples $(n;\un{a} | \un{b})$, where
$|\un{a}| \le n$ and $|\un{b}| \le n$, and thereby to the corresponding type-{\sf D} seaweeds and 
meander graphs. For instance, the first equality in~\eqref{q-shtrikh} means that we replace 
$(n;\un{a} | \un{b})$ with $(n-a_1;a_2,\dots,a_s | b_1-2a_1,a_1,b_2,\dots,b_t)$, if $a_1\le b_1/2$.
Accordingly, $\Gamma_n(\un{a}|\un{b})$ is replaced with 
$\Gamma_{n-a_1}(a_2,\dots,a_s{\mid} b_1-2a_1,a_1,b_2,\dots,b_t)$. An important feature is that 
{\sl\bfseries Step~2}{\sf (ii)} may (and will) be understood in the graph setting as the contraction of certain arcs in $\Gamma_n(\un{a}|\un{b})$ related to the parts $a_1,b_1$, see~\cite[Lemma\,5.4(i)]{MY12}. 
Since $\Gamma_n(\un{a}|\un{b})$ is symmetric w.r.t. the $\sigma$-mirror, these contractions are performed 
simultaneously on the both ends of it. 
The pictures below demonstrate the effect of contractions in the left hand end of the meander graph
$\Gamma_n(\un{a}|\un{b})$.
\vskip1ex
\centerline{\sl The case in which \ $a_1\le b_1/2$ :}

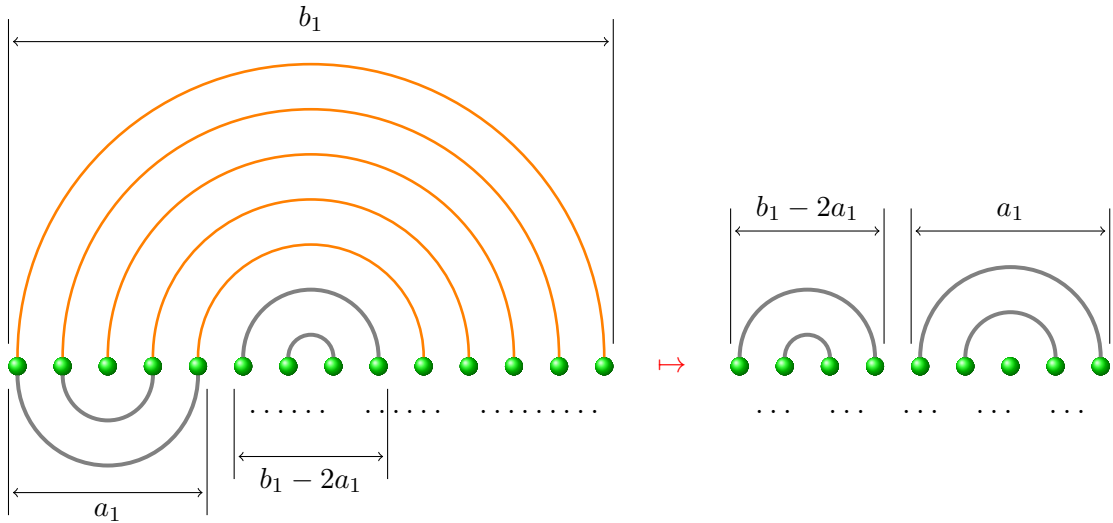
\begin{center}
\begin{tikzpicture}[scale= .6] 
\path  (0,0) -- (25,0);
\foreach \x in {0,...,13}  \shade[shading=ball, ball color=green](\x,0) circle (2mm); 
\draw[line width=1.5pt,color=gray] (0, -0.2)   arc(180:360:2cm);
\draw[line width=1.5pt,color=gray] (1, -0.2)   arc(180:360:1cm);
\draw[line width=1pt,color=orange] (13, 0.2)   arc(0:180:6.5cm);
\draw[line width=1pt,color=orange] (12, 0.2)   arc(0:180:5.5cm);
\draw[line width=1pt,color=orange] (11, 0.2)   arc(0:180:4.5cm);
\draw[line width=1pt,color=orange] (10, 0.2)   arc(0:180:3.5cm);
\draw[line width=1pt,color=orange] (9, 0.2)   arc(0:180:2.5cm);
\draw[line width=1.5pt,color=gray] (8, 0.2)   arc(0:180:1.5cm);
\draw[line width=1.5pt,color=gray] (7, 0.2)   arc(0:180:.5cm);
\draw (9,-1)  node {$\mathbf{\dots\dots \quad \dots \dots \quad \dots \dots \dots}$};
\draw (-0.2, -.5) -- (-0.2,-3.3);
\draw (4.2, -.5) -- (4.2,-3.3);
\draw[<->] (-.1,-2.8) -- (4.1,-2.8) node[pos=.5,below] {\small $a_1$};
\draw (4.8, -0.5) -- (4.8,-2.5);
\draw (8.2, -0.5) -- (8.2,-2.5);
\draw[<->] (4.9,-2) -- (8.1,-2) node[pos=.5,below] {\small $b_1-2a_1$};
\draw (-0.2, .5) -- (-0.2,7.7);
\draw (13.2, .5) -- (13.2,7.7);
\draw[<->] (-.1,7.2) -- (13.1,7.2) node[pos=.5,above] {\small $b_1$};

\foreach \x in {16,...,24}  \shade[shading=ball, ball color=green](\x,0) circle (2mm); 
\draw[line width=1.5pt,color=gray] (19, 0.2)   arc(0:180:1.5cm);
\draw[line width=1.5pt,color=gray] (18, 0.2)   arc(0:180:.5cm);
\draw[line width=1.5pt,color=gray] (24, 0.2)   arc(0:180:2cm);
\draw[line width=1.5pt,color=gray] (23, 0.2)   arc(0:180:1cm);
\draw (20,-1)  node {$\mathbf{\dots\quad \dots \quad \dots \quad \dots \quad \dots}$};
\draw (15.8, .5) -- (15.8,3.5);
\draw (19.2, .5) -- (19.2,3.5);
\draw[<->] (15.9,3) -- (19.1,3) node[pos=.5,above] {\small $b_1-2a_1$};
\draw (19.8, .5) -- (19.8,3.5);
\draw (24.2, .5) -- (24.2,3.5);
\draw[<->] (19.9,3) -- (24.1,3) node[pos=.5,above] {\small $a_1$};

\draw (14.5,0)  node[color=red]  {\small $\mapsto$};
\end{tikzpicture}
\end{center}

\vskip2ex
\centerline{\sl The case in which \ $b_1/2\le a_1 < b_1$ :}

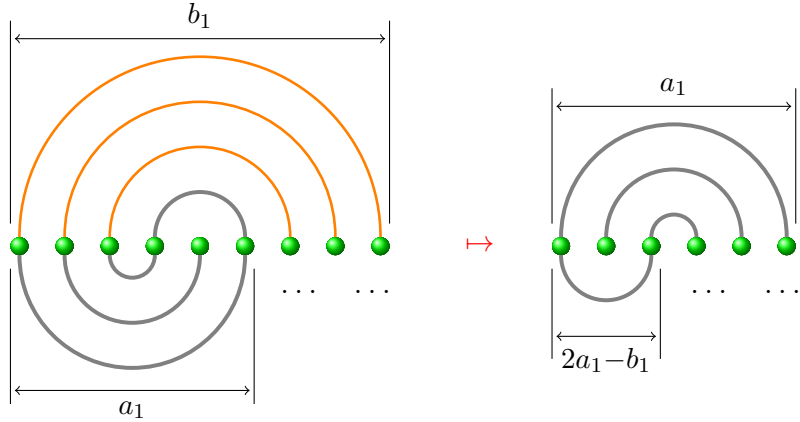
\begin{center}
\begin{tikzpicture}[scale= .6] 
\path  (0,0) -- (18,0);
\foreach \x in {0,...,8}  \shade[shading=ball, ball color=green](\x,0) circle (2mm); 
\draw[line width=1.5pt,color=gray] (0, -0.2)   arc(180:360:2.5cm);
\draw[line width=1.5pt,color=gray] (1, -0.2)   arc(180:360:1.5cm);
\draw[line width=1.5pt,color=gray] (2, -0.2)   arc(180:360:.5cm);
\draw[line width=1pt,color=orange] (8, 0.2)   arc(0:180:4cm);
\draw[line width=1pt,color=orange] (7, 0.2)   arc(0:180:3cm);
\draw[line width=1pt,color=orange] (6, 0.2)   arc(0:180:2cm);
\draw[line width=1.5pt,color=gray] (5, 0.2)   arc(0:180:1cm);
\draw (7,-1)  node {$\mathbf{\dots\quad \dots}$};
\draw (-0.2, -.5) -- (-0.2,-3.5);
\draw (5.2, -.5) -- (5.2,-3.5);
\draw[<->] (-.1,-3.2) -- (5.1,-3.2) node[pos=.5,below] {\small $a_1$};
\draw (-0.2, .5) -- (-0.2,5);
\draw (8.2, .5) -- (8.2,5);
\draw[<->] (-.1,4.6) -- (8.1,4.6) node[pos=.5,above] {\small $b_1$};

\foreach \x in {12,...,17}  \shade[shading=ball, ball color=green](\x,0) circle (2mm); 
\draw[line width=1.5pt,color=gray] (12, -0.2)   arc(180:360:1cm);
\draw[line width=1.5pt,color=gray] (17, 0.2)   arc(0:180:2.5cm);
\draw[line width=1.5pt,color=gray] (16, 0.2)   arc(0:180:1.5cm);
\draw[line width=1.5pt,color=gray] (15, 0.2)   arc(0:180:.5cm);
\draw (11.8, -.5) -- (11.8,-2.5);
\draw (14.2, -.5) -- (14.2,-2.5);
\draw[<->] (11.9,-2) -- (14.1,-2) node[pos=.5,below] {\small $2a_1{-}b_1$};
\draw (11.8, .5) -- (11.8,3.5);
\draw (17.2, .5) -- (17.2,3.5);
\draw[<->] (11.9,3.1) -- (17.1,3.1) node[pos=.5,above] {\small $a_1$};
\draw (16,-1)  node {$\mathbf{\ \dots\quad \dots}$};

\draw (10.2,0)  node[color=red]  {\small $\mapsto$};
\end{tikzpicture}
\end{center}

In each case, we contract the {\color{orange}orange} arcs to the right end points, and the whole 
configuration including the grey arcs meeting the first $a_1$ nodes is rotated clockwise through the angle 
180 degrees about the middle point of the first  $b_1$ vertices. We do not draw the vertices after $b_1$ 
and the arcs related to the parts $a_2,b_2$, etc., because the corresponding fragments of the meander 
graph remain intact.

A subtle point is that after a certain contraction applied to a graph with crossing, one can obtain a graph with crossing on the wrong side. It will be explained below how to handle this situation.
\end{rmk}
We say that a seaweed $\q$ {\it reduces to zero\/} if after some inductive step we obtain $\q'=0$.
(This happens if and only if at the previous stage one has $\q=\q_m(m|m)\simeq \gt{gl}_m$, and 
{\sl\bfseries Step~2}(i) applies.) The corresponding meander graph is empty. 

\begin{proof}[Proof of Theorem~\ref{thm:main}]
We use the above {\sf \bfseries inductive procedure}, which is understood as a procedure 
applied simultaneously to seaweeds and their meander graphs, see Remark~\ref{rmk:auch-for-graphs}.  
Given a seaweed $\q\subset\sone$, consider its type-{\sf D} meander graph $\Gamma_n(\q)$ and set 
\[
 \cT(\q):=\#\{\text{the cycles of } \Gamma_n(\q)\} + 
   \frac{1}{2}\#\{\text{the non-$\sigma$-stable segments of $\Gamma_n(\q)$}\} .
\] 
Let us prove that $\ind\q$ and $\cT(\q)+\epsilon(\q)$ \ behave accordingly for 
{\sl\bfseries Steps~2}{\sf (i)} and {\sl\bfseries 2}{\sf (ii)}. 

If $a_1=b_1$, then $\q=\gt{gl}_{a_1}\oplus \q'$, where $\q'\subset \gt{so}_{2(n-a_1)}$, and 
$\ind\q'=\ind\q- a_1$.  On the other hand, $\Gamma_{n-a_1}(\q')$  is obtained from $\Gamma_n(\q)$ by
deleting $2\left[ \frac{a_1}{2} \right ]$ cycles (and two segments, which are not 
$\sigma$-invariant, if $a_1$ is odd). Therefore $\cT(\q')=\cT(\q)-a_1$ and also $\epsilon(\q)=\epsilon(\q')$. 

If $a_1\ne b_1$, then $\ind\q'=\ind\q$. Basically, {\sl\bfseries Step~2}(ii) in type {\sf D} (also for the 
seaweeds with crossing) consists of two ``symmetric'' type-{\sf A} reductions applied simultaneously  
to the both ends of $\Gamma_n(\q)$. 
On the graph level, this step is interpreted as contraction of certain non-central edges, cf. the 
above pictures.
Therefore, this does not change the topological structure of the graph and the number of central edges.
Hence $\cT(\q')=\cT(\q)$ and $\epsilon$ also has the same value for $\q'$ and $\q$. 

\textbullet\quad If $\q$ has no crossing, then the {\sf\bfseries procedure} is being repeated until we end up 
with a parabolic subalgebra. This settles the problem for the seaweeds without crossing.

\textbullet\quad Suppose now that $\q$ has a crossing and $\q=\q({\un{a}}|\un{b})_{\sf c}$ with 
${\un{a}}=(a_1,\dots,a_s)$ and $\un{b}=(b_1,\dots,b_t)$, as explained in 
Section~\ref{subs:with}. Recall that  then $|{\un{a}}|=|\un{b}|=n$ and $a_s, b_t\ge 2$.
\\ \indent
Suppose that $a_s\le b_t$ and hence the crossing is  below the horizontal line. Then one can apply  
{\sl\bfseries Step~2} as long as the second composition has at least two parts. This eventually kills all the parts $b_i$ with $i<t$ and provides 
the situation, where $\un{b}=(b_t)$. So, let us assume that $t=1$, $n=b_1\ge 2$ and 
${\un{a}}=(a_1,\dots,a_s)$, as above.  If $s=1$, then $\q=\q_{\sf ec}(n)$, and the reduction terminates. If $s\ge 2$, 
then one can still apply {\sl\bfseries Step~2}(ii) to $\q=\q_n(a_1,\dots,a_s| n)_{\sf c}$.  
This replaces $\un{b}=(n)$ with another {\sl second\/} composition $\un{b}'$.  
By~\eqref{q-shtrikh}, we have 
$\un{b}'=\left\{ \begin{array}{cl}  (a_1), & \text{if } \  a_1\ge n/2 \\  
(n-2a_1,a_1), &\text{if } \    a_1\le n/2  \end{array} \right .$. \quad
That is, the last part of $\un{b}'$ is always $a_1$, while the last part of the new {\sl first\/} composition 
$\un{a}'$ is always $a_s$. If $a_1\ge a_s$, then the corresponding contraction yields the graph with 
crossing on the correct side. Then we continue the {\sf\bfseries procedure} with $\un{a'},\un{b}'$. 
If $a_1< a_s$, then $a_1< n/2$ and the passage 
\[
(a_1,\dots,a_s| n) \leadsto
(a_2,\dots,a_s| n-2a_1,a_1)=(\un{a}'|\un{b}') 
\] 
suggests that we should have obtained a meander graph with crossing {\bf above} the horizontal line.
But the contraction of edges in $\Gamma_n(\un{a}|\un{b})_{\sf c}$ yields a graph $\tilde\Gamma$ with crossing {\bf below} the horizontal line, as it was; i.e., crossing is now on the wrong side! 

To remedy this, we permute two central vertices in $\tilde\Gamma$, which merely corresponds to 
the permutation of two basis vectors in the space $\BC^{2(n-a_1)}$ of the standard representation of
$\gt{so}_{2(n-a_1)}$.
This does not change the topological structure of $\Gamma'$ and provides the meander graph, 
$\Gamma'$, of $\q'$. There are two possibilities, though. If $\Gamma'$ still has a crossing (already on the 
{\sl correct\/} side!), then we resume the {\sf\bfseries procedure}. The alternative possibility is that the crossing vanishes. This can only happen if we had two strange components 
(segments). More precisely, this happens if and only if the last part of $\un{b}'$, i.e. $a_1$, is equal to $1$, 
cf. Example~\ref{ex:chain-uncross}. In both 
cases, the value of $\epsilon$ does not change. (In the second case, we have $\epsilon=0$ before and after the permutation.) 
\\ \indent
Thus, either the crossing vanishes at some stage and the seaweed eventually reduces to zero, or 
the reduction terminates with $s=t=1$ and $a_1=b_1=:m$, i.e., $\q=\q_{\sf ec}(m)$.

Since Eq.~\eqref{eq:index-D} is already verified for the parabolic subalgebras and all $\q_{\sf ec}$ by 
Lemma~\ref{Ind-base}, the result follows. 
\end{proof}

\begin{rmk}    
The last part of the proof of Theorem~\ref{thm:main} shows that there are two alternatives for the 
{\sf\bfseries inductive procedure} applied to the seaweeds with crossing. Either $\q$ has one strange component (cycle), and then it reduces to some $\q_{\sf ec}(m)$; or $\q$ has two strange components
(segments), and then the crossing eventually vanishes and $\q$ reduces to zero.
\end{rmk}
\begin{ex}           \label{ex:chain-uncross}
Let us apply the inductive procedure to the seaweed $\q_5(2,3|1,4)_{\sf c}$, see Fig.~\ref{pikcha:D5-2}. 
The chain of seaweeds and reduction steps is
\[ 
\q_5(2,3|1,4)_{\sf c} \leadsto \q_4(1,3|4)_{\sf c} \leadsto 
\q_3(3|2,1)_{\sf c} \stackrel{\ast}{\leadsto} \q_3(3|2,1) \leadsto \\
\q_2(2|1,1) \leadsto \q_1(1|1) \leadsto 0 \ .
\] 
That is, this seaweed reduces to zero. The second step gives us the graph $\Gamma_3(3|2,1)_{\sf c}$
with crossing on the wrong side. Therefore, the next
step marked with the asterisk is the permutation of vertices
$n$ and $n+1$ (with $n=3$), which results in disappearance of crossing.
The corresponding chain of meander graphs is depicted in Fig.~\ref{pikcha:chain}. The edge(s) that are going to be contracted on the next step are depicted in {\color{orange}orange}.
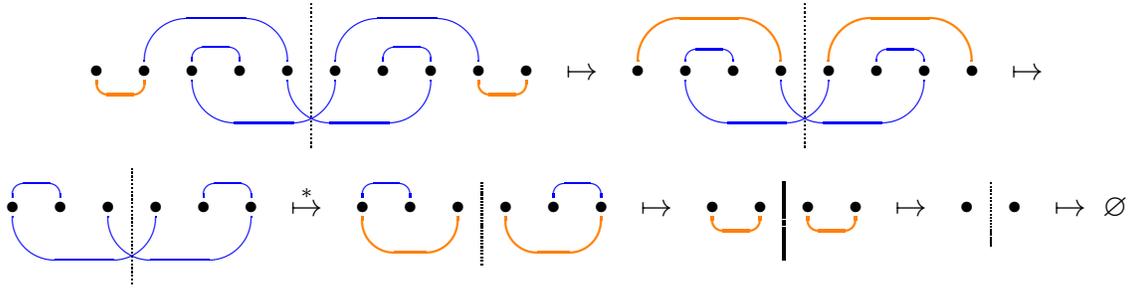
\begin{figure}[htb]
\setlength{\unitlength}{0.025in}
\begin{center}
\begin{picture}(100,25)(15,-2)

\multiput(20,3)(10,0){10}{\circle*{2}}

{\color{blue}
\put(45,5){\oval(10,6)[t]}
\put(45,5){\oval(30,18)[t]}
\put(85,5){\oval(10,6)[t]}
\put(85,5){\oval(30,18)[t]}

\put(55,1){\oval(30,18)[b]}
\put(75,1){\oval(30,18)[b]}
}
\qbezier[33](65,-13),(65,2),(65,17)
\thicklines
{\color{orange}
\put(25,1){\oval(10,6)[b]}
\put(105,1){\oval(10,6)[b]}
}
\end{picture} 
\ \raisebox{1ex}{$\mapsto$} \ 
\begin{picture}(80,25)(25,-2)

\multiput(30,3)(10,0){8}{\circle*{2}}

{\color{blue}
\put(45,5){\oval(10,5)[t]}
\put(85,5){\oval(10,5)[t]}

\put(55,1){\oval(30,18)[b]}
\put(75,1){\oval(30,18)[b]}
}
\qbezier[33](65,-13),(65,2),(65,17)
\thicklines
{\color{orange}
\put(45,5){\oval(30,18)[t]}
\put(85,5){\oval(30,18)[t]}
}
\end{picture} 
\ \raisebox{1ex}{$\mapsto$} \\
\begin{picture}(60,28)(35,-2)

\multiput(40,3)(10,0){6}{\circle*{2}}

{\color{blue}
\put(45,5){\oval(10,6)[t]}
\put(85,5){\oval(10,6)[t]}

\put(55,1){\oval(30,18)[b]}
\put(75,1){\oval(30,18)[b]}
}
\qbezier[30](65,-13),(65,-1),(65,11)
\end{picture} 
\ \raisebox{1ex}{$\stackrel{\ast}{\mapsto}$} \
\begin{picture}(60,28)(35,-2)

\multiput(40,3)(10,0){6}{\circle*{2}}

{\color{blue}
\put(45,5){\oval(10,6)[t]}
\put(85,5){\oval(10,6)[t]}

}
\qbezier[25](65,-9),(65,0),(65,8)
\thicklines
{\color{orange}
\put(50,1){\oval(20,15)[b]}
\put(80,1){\oval(20,15)[b]}
}
\end{picture} 
\ \raisebox{1ex}{$\mapsto$} \ 
\begin{picture}(40,28)(45,-2)

\multiput(50,3)(10,0){4}{\circle*{2}}
\thicklines
{\color{orange}
\put(55,1){\oval(10,6)[b]}
\put(75,1){\oval(10,6)[b]}
}
\qbezier[22](65,-8),(65,0),(65,8)
\end{picture} 
\ \raisebox{1ex}{$\mapsto$} \ 
\begin{picture}(20,28)(55,-2)

\multiput(60,3)(10,0){2}{\circle*{2}}

\qbezier[20](65,-5),(65,0),(65,8)
\end{picture} 
\ \raisebox{1ex}{$\mapsto \ \varnothing$}
\end{center}
\caption{The reduction steps for $\Gamma_5(2,3|1,4)_{\sf c}$}   \label{pikcha:chain}
\end{figure}

For all graphs here, one has $\epsilon=0$.
\end{ex}

\section{Miscellaneous remarks and applications}
\label{sect:applic}

\subsection{Generic stabilisers}
Given a Lie algebra $\rr$, write $\rr_\ast$ for a generic stabiliser of the coadjoint representation
$(\rr,\ads)$, if it exists. For any seaweed $\q$ in types {\sf A} and {\sf C},  
a generic stabiliser $\q_\ast$ exists; 
moreover, it is a torus, see~\cite{Dima03-b}. 
For all other simple Lie algebras, there are parabolic subalgebras $\p$ such that the coadjoint representation has no generic stabilisers, 
see \cite[3.2]{ty-jA} for $\mathsf D_4$ and \cite[Section\,6]{Dima03-b} in general. 

It is shown in~\cite{Dima03-b} that the {\sf\bfseries inductive procedure} of Section~\ref{sect:main}
can be used for computing a generic stabiliser and proving its existence. 
In {\sl\bfseries Step~2}(i), for a generic $\beta\in\gt q^*$, we have 
$\q_\beta=\q'_{\beta'}\oplus \te_{a_1}$, where
$\te_{a_1}$ is a maximal torus in $\gt{gl}_{a_1}$ and $\beta'$ is the restriction of $\beta$ to 
$\q'$.  Therefore, if $\q'_{\beta'}$ is a generic stabiliser for $(\q',\ads)$, then $\q_\beta$ is a generic stabiliser for $(\q,\ads)$.
In {\sl\bfseries Step~2}(ii), the situation is even better. If $(\q',\ads)$ has a generic stabiliser, 
say $\q'_\ast$, then $\q'_\ast$ is a generic stabiliser for $(\q,\ads)$, too. 

In type {\sf D}, a seaweed $\q$ without crossing reduces to a parabolic subalgebra $\p$. Hence
$\q_\ast$ exists and is a torus if and only if $\p_\ast$ exists and is a torus. 
The parabolic subalgebras  $\p\subset\gt{so}_m$ such that $\p_\ast$ is reductive (and therefore is a torus) 
are classified in~\cite[Th\'eor\`eme\,29]{DKT}, cf. also~\cite[Lemma~2.3 \& Def.~5.7]{MY12}.

For $\q_{\sf ec}(n)$, a generic stabiliser for  the coadjoint action is a torus of dimension $n-2$. 
And a seaweed with crossing having two strange components reduces to zero as can be seen from 
the proof of Theorem~\ref{thm:main}. Thus, we happily arrive at the following conclusion. 

\begin{prop}\label{prop-red-st} 
A seaweed with crossing possesses a non-empty open subset $U\subset \gt q^*$ such 
that $\gt q_\beta$ is a torus for each $\beta\in U$ and all these stabilisers are conjugate by elements 
of the connected group $\exp(\gt q)=Q$, i.e., $\gt q_\ast=\gt q_\beta$. 
\end{prop}

\subsection{Strongly quasi-reductive seaweeds}
Following~\cite{MY12}, a Lie algebra $\q=\Lie Q$ is said to be {\it strongly quasi-reductive\/} if there is a
$\gamma\in\gt q^*$ such that $Q_\gamma$ is reductive. 
A more general notion of {\it quasi-reductive\/} Lie algebras is considered in \cite{DKT}. However, these two 
coincide for the seaweed subalgebras, since the centre of any $\q(S,T)$ consists of semisimple elements.

By \cite[Th\'eor\`eme\,9]{DKT}, if $\q$ is strongly quasi-reductive, 
then there is a reductive stabiliser $Q_\gamma$ (with $\gamma\in \q^*$) such that, up to conjugation, 
any other reductive stabiliser $Q_\beta$ (with $\beta\in\gt q^*$) is contained in $Q_\gamma$. 
In~\cite{MY12}, such a group $Q_\gamma$ is  called a {\it maximal reductive stabiliser} of $\q$, \MRS 
for short. For a seaweed  $\q=\q^{\sf A}(\un{a}|\un{b})\subset\gln$, an \MRS of $\q$ can be described in 
terms of $\Gamma^{\sf A}(\un{a}|\un{b})$~\cite[Theorem~5.3]{MY12}. 
In particular, an \MRS of $\gt q$ is isomorphic to $\GL_2$ if and only if 
$\Gamma^{\sf A}(\un{a}|\un{b})$ is a single cycle. 

For any seaweed with crossing, there is $\beta\in \q^*$ such that $\q_\beta$ is a torus (Proposition~\ref{prop-red-st}). Therefore, the seaweeds with crossing are strongly quasi-reductive. All seaweeds $\q$ in type {\sf A} or {\sf C} are also strongly quasi-reductive for the same reason, 
$\q_\ast$ is a torus~\cite{Dima03-b}. 


\subsection{Frobenius seaweeds in type {\sf D}}
A Lie algebra $\q$ is said to be {\it Frobenius} if $\ind\q=0$. 
Such Lie algebras are quite popular nowadays. Frobenius seaweeds in type ${\sf A}_n$ are rather 
mysterious. Even the asymptotic behaviour of their distribution remains unknown. Partial results on  
the Frobenius seaweeds in type {\sf C} are obtained in \cite{SW-C}.  Let us see what happens in
type {\sf D}. 
  
\begin{prop}    \label{prop:fro-cr}
Let $\q\subset\sone$ be a seaweed \emph{with} crossing. If\/ $\Gamma_n(\q)$ has two strange components,
then $\q$ cannot be Frobenius. If\/ $\Gamma_n(\q)$ has one strange component and $\ind\q=0$, then
$\Gamma_n(\q)$ is a single cycle and $n$ is even. Moreover,
there is a bijection between the standard Frobenius seaweeds $\q$
with crossing (up to the transposition $\alpha_{n-1}\longleftrightarrow\alpha_n$) and 
the standard seaweeds $\es\subset\gln$ such that an\/ \MRS of\/ $\es$ is $\GL_2$.     
\end{prop}
\begin{proof}
If $\q\subset\sone$ has a crossing and $\Gamma_n(\q)$ has two strange connected components, then 
$\epsilon=0$ and there are at least two segments that are not $\sigma$-stable. Therefore $\ind\q>0$ regardless of the parity of $n$.

Suppose that $\q=\q(\un{a}|\un{b})_{\sf c}$, the strange component of 
$\Gamma(\un{a}|\un{b})_{\sf c}$ is a cycle, and $\ind\q=0$. The meander graph of a seaweed with crossing
has no $\sigma$-stable segments (Section~\ref{subs:with}). Therefore, by Eq.~\eqref{eq:index-D}, this 
strange cycle must be the only component of $\Gamma(\un{a}|\un{b})_{\sf c}$.
It is then easily seen that $\es:=\q^{\sf A}(\un{a}|\un{b})\subset\gln$ has the property that 
$\Gamma^{\sf A}(\un{a}|\un{b})$ is a single cycle and therefore an \MRS of $\es$ is isomorphic to  $\GL_2$.
If we invoke the ``three step'' construction of $\Gamma_n(\q)=\Gamma(\un{a}|\un{b})_{\sf c}$ in 
Section~\ref{subs:with}, then $\Gamma^{\sf A}(\un{a}|\un{b})$ represents the left hand side half of the graph
$\Gamma_n(\check\q)$ obtained in Step~{\sf 2)}.
Conversely, if  $\un{a}$ and $\un{b}$ are two compositions of $n$ such that 
$\Gamma^{\sf A}(\un{a}|\un{b})$ is a single cycle, then so is 
$\Gamma(\un{a}|\un{b})_{\sf c}$ (for $\q=\q(\un{a}|\un{b})_{\sf c}\subset\sone$). 
See a sample in Figure~\ref{pikcha:D6-1} below. 
It remains to observe that if $\Gamma^{\sf A}(\un{a}|\un{b})$ is a single cycle, then $n$ is necessarily even.
\end{proof}

\begin{ex}
Let $\q=\q(S,T)\subset\gt{so}_{12}$ with $S=\{\alpha_1, \alpha_3,\alpha_4, \alpha_6\}$, 
$T=\{\alpha_1, \alpha_2,\alpha_3, \alpha_4, \alpha_5\}$. 
Then $\un{a}=(2,4)$, $\un{b}=(6)$ and $\Gamma_{6}(\q)$ 
is a single strange cycle. Hence $\ind\q=0$. 
To illustrate the bijection of Proposition~\ref{prop:fro-cr}, we also draw the graph 
$\Gamma^{\sf A}(2,4|6)$ in Figure~\ref{pikcha:D6-1}.  

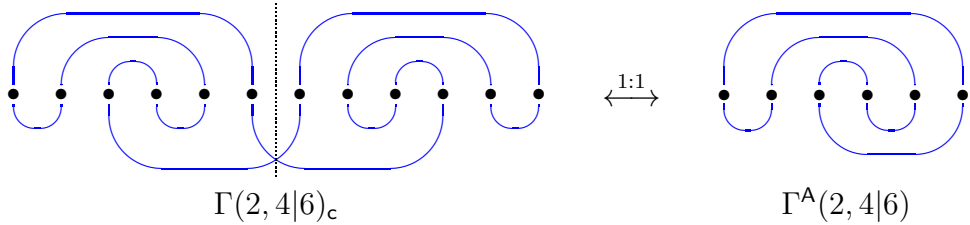
\begin{figure}[htb]
\setlength{\unitlength}{0.025in}
\begin{center}
\begin{picture}(110,35)(25,-15)

\put(62,-22){$\Gamma(2,4|6)_{\sf c}$}
\multiput(20,3)(10,0){12}{\circle*{2}}

{\color{blue}
\put(45,5){\oval(50,30)[t]}
\put(105,5){\oval(50,30)[t]}
\put(45,5){\oval(30,20)[t]}
\put(105,5){\oval(30,20)[t]}
\put(45,5){\oval(10,10)[t]}
\put(105,5){\oval(10,10)[t]}

\put(55,1){\oval(10,10)[b]}
\put(25,1){\oval(10,10)[b]}%
\put(60,1){\oval(40,27)[b]}%
\put(125,1){\oval(10,10)[b]}%
\put(95,1){\oval(10,10)[b]}
\put(90,1){\oval(40,27)[b]}%
}
\qbezier[42](75,-14),(75,4),(75,22)
\end{picture}
\quad \raisebox{5ex}{$\stackrel{1:1}{\longleftrightarrow}$} \quad
\begin{picture}(60,35)(15,-15)

\put(32,-22){$\Gamma^{\sf A}(2,4|6)$}
\multiput(20,3)(10,0){6}{\circle*{2}}

{\color{blue}
\put(45,5){\oval(50,30)[t]}
\put(45,5){\oval(30,20)[t]}
\put(45,5){\oval(10,10)[t]}

\put(55,1){\oval(30,21)[b]}
\put(55,1){\oval(10,11)[b]}
\put(25,1){\oval(10,11)[b]}
}
\end{picture}
\end{center}
\vspace{1ex}
\caption{The bijection of Proposition~\ref{prop:fro-cr}}
\label{pikcha:D6-1}
\end{figure}
\end{ex}


\begin{ex}    \label{ex:tri-grafa}
The Lie algebra $\gt{so}_8$ has three different non-equivalent matrix realisations 
($8$-dimensional representations) corresponding to the fundamental weights $\varpi_1,\varpi_3$, and 
$\varpi_4$. Therefore, each seaweed acquires three (usually different) meander graphs. Yet, 
\eqref{eq:index-D} gives the same value for all possible graphs. \\ \indent 
Consider $\q=\q(S,T)\in\gt{so}_8$ with $S=\{\alpha_1,\alpha_3\}$ and $T=\{\alpha_1,\alpha_2,\alpha_4\}$. 
Then
\\ \indent
\textbullet \quad for the realisation associated with $\varpi_1$, $\q$ has a crossing; more precisely,
$\q=\q(2,2|4)_{\sf c}$;

\textbullet \quad for $\varpi_3$, there is no crossing and %
$\q=\q_4(2,2|1)$; 

\textbullet \quad for $\varpi_4$, one obtains $\q=\q_4(1,1|4)$.  
\\
The corresponding meander graphs are presented in~Fig.~\ref{pikcha:so(8)}.
\begin{figure}[htb]
\setlength{\unitlength}{0.025in}
\begin{center}
\begin{picture}(75,30)(17,-8)
\put(46,20){for $\varpi_1$}
\multiput(20,3)(10,0){8}{\circle*{2}}

{\color{blue}
\put(35,5){\oval(10,7)[t]}
\put(35,5){\oval(30,17)[t]}
\put(75,5){\oval(10,7)[t]}
\put(75,5){\oval(30,17)[t]}

\put(50,1){\oval(20,17)[b]}
\put(25,1){\oval(10,7)[b]}
\put(60,1){\oval(20,17)[b]}
\put(85,1){\oval(10,7)[b]}

\qbezier[45](55,-15),(55,2),(55,20)}
\end{picture}
\quad 
\begin{picture}(75,30)(17,-8)
\put(46,20){for $\varpi_3$}
\multiput(20,3)(10,0){8}{\circle*{2}}

{\color{blue}
\put(55,5){\oval(10,7)[t]}
\put(55,5){\oval(30,17)[t]}
\put(55,5){\oval(50,22)[t]}

\put(45,1){\oval(10,7)[b]}
\put(25,1){\oval(10,7)[b]}
\put(65,1){\oval(10,7)[b]}
\put(85,1){\oval(10,7)[b]}

\qbezier[45](55,-15),(55,2),(55,20)}
\end{picture}
\quad 
\begin{picture}(75,30)(17,-8)
\put(46,20){for $\varpi_4$}
\multiput(20,3)(10,0){8}{\circle*{2}}

{\color{blue}
\put(35,5){\oval(10,7)[t]}
\put(35,5){\oval(30,17)[t]}
\put(75,5){\oval(10,7)[t]}
\put(75,5){\oval(30,17)[t]}

\put(55,1){\oval(30,17)[b]}
\put(55,1){\oval(10,7)[b]}

\qbezier[45](55,-15),(55,2),(55,20)}
\end{picture}
\end{center}
\caption{Three meander graphs for one seaweed in $\gt{so}_8$}   \label{pikcha:so(8)}
\end{figure}
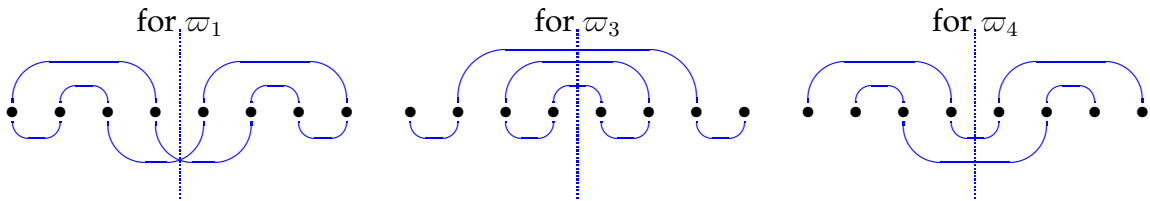

One readily verifies that Theorem~\ref{thm:main} yields $\ind\q=0$ for all three graphs.
\end{ex}


\begin{prop}           \label{prop:frob-no-cr}
Let $\q\subset\sone$ be a standard seaweed \emph{without} crossing.
\\ \indent
{\sf (i)} If\/ $\q$ is Frobenius, then $\epsilon=\epsilon(\q)\in \{0,-1\}$.   
\\ \indent
{\sf (ii)} There is a bijection between the standard Frobenius seaweeds $\q$ 
such that $\epsilon=0$ 
and the standard Frobenius seaweeds in $\spn$ having an even number of central arcs.  
\\ \indent
{\sf (iii)} If $\epsilon=-1$, then $\q$ is Frobenius  if and only if\/  
$\Gamma_n(\q)$ has an odd number of central arcs, all of which are on one and the same side of the horizontal line, 
exactly one cycle going through the vertices $n$ and $n+1$, and no segments that are not $\sigma$-stable. 
\end{prop}
\begin{proof}  
{\sf (i)} 
If $\q$ is Frobenius, then Eq.~\eqref{eq:index-D} shows that $\epsilon=1$ is not allowed. 

{\sf (ii)}  If $\epsilon=0$ and $\ind\gt q=0$, then $\Gamma_n(\q)$ has 
no cycles, all its segments are $\sigma$-stable, and $m_b-m_a$ is even,
see  Eq.~\eqref{eq:index-D}.  Assume that $m_b\ge m_a$. 
If $m_a>0$, then $\Gamma_n(\q)$ has a cycle, a contradiction! Hence $m_a=0$ and $m_b$ is even. 
The graph $\Gamma_n(\q)$ can also be regarded as the type-{\sf C} meander graph of a seaweed 
$\check{\q}\subset\spn$, and $\ind\q=0$ if and only if $\ind{\check{\q}}=0$,
cf. Theorem~\ref{thm:main} and~\cite[Theorem\,3.2]{SW-C}.
  
({\sf iii})  If $\epsilon=-1$ and  $\ind\q=0$, then $m_b-m_a$ is odd and $\Gamma_n(\q)$ has either a single cycle or two non $\sigma$-stable segments. Assume that $m_b\ge m_a$. If $m_a\ge 2$, then 
$\Gamma_n(\q)$ contains at least two cycles and $\ind\gt q>0$, a contradiction! If $m_a=1$, then 
$n-|\un{a}|=1$, which is not allowed, see Remark~\ref{rmk:exclude}. 
Hence $m_a=0$ and $m_b$ is odd. Since $\epsilon\ne 1$, the central arc between the vertices $n$ 
and $n+1$ belongs to a cycle, which is the unique cycle in $\Gamma_n(\q)$. Hence all the segments must be 
$\sigma$-stable. It remains to observe that this argument can be reversed.
\end{proof}

\end{document}